%% file: Main_v15.tex
\documentclass[11pt,twoside]{amsart}

\input{preambule}
\usepackage{amsmath,amsfonts,amssymb,amsthm}
\usepackage[all,cmtip]{xy}
\usepackage{tikz-cd}
\usepackage[
backend = biber,
citestyle = alphabetic,
style = alphabetic,
maxnames = 50,
giveninits=true,
]{biblatex}
\usepackage[colorlinks]{hyperref}
\hypersetup{
    colorlinks,
    citecolor=magenta,
    filecolor=red,
    linkcolor=blue,
    urlcolor=purple
}
\usepackage{mathtools}
\usepackage{cleveref}
\usepackage[T2A]{fontenc}
\DeclareSymbolFont{cyrillic}{T2A}{cmr}{m}{n}
\DeclareMathSymbol{\Sha}{\mathalpha}{cyrillic}{216}
\usepackage{xcolor}
 
\usepackage[normalem]{ulem}

\numberwithin{equation}{section}

\newtheorem{dummy}{dummy}[section]

\newtheorem{definition}[dummy]{Definition}
\newtheorem{theorem}[dummy]{Theorem}
\newtheorem{corollary}[dummy]{Corollary}
\newtheorem{lemma}[dummy]{Lemma}
\newtheorem{proposition}[dummy]{Proposition}
\theoremstyle{definition}
\newtheorem{remark}[dummy]{Remark}

\theoremstyle{plain}
\newtheorem{defprop}[dummy]{Definition/Proposition}

\newcommand{\Z}{\ensuremath{\mathbb{Z}}}

\newcommand{\spec}{\ensuremath{\operatorname{Spec}}}
\newcommand{\Q}{\ensuremath{\mathbb{Q}}}
\newcommand{\CC}{\ensuremath{\mathbb{C}}}
\newcommand{\PP}{\ensuremath{\mathbb{P}}}
\newcommand{\bigo}{\ensuremath{\mathcal{O}}}
\newcommand{\klat}{\ensuremath{\Lambda_{\operatorname{K3}}}}

\def\Db{\calD^{b}}
\newcommand{\catname}[1]{{\normalfont\textbf{#1}}}

\def\NS{\operatorname{NS}}
\def\Pic{\operatorname{Pic}}
\def\olPic{\overline{\operatorname{Pic}}}

\def\Br{\operatorname{Br}}

\def\rk{\operatorname{rk}}

\def\disc{\operatorname{disc}}
\def\Jac{\operatorname{J}}
\def\Kn{{\ensuremath{\operatorname{K3}^{[n]}}}}

\def\ol{\overline}

\def\N{\operatorname{N}}
\def\SBr{\operatorname{SBr}}
\def\toplus{\ensuremath{\tilde{O}^+}}

\def\dv{\operatorname{div}}

\title{Obstruction Classes for Moduli Spaces of Sheaves and Lagrangian Fibrations}
\author[D.~Mattei]{Dominique Mattei}
\address{
Mathematisches Institut \& Hausdorff Center for Mathematics,
Universit\"at Bonn, Endenicher Allee 60, 53115 Bonn, Germany.
}
\email{dmattei@math.uni-bonn.de}
\author[R.~Meinsma]{Reinder Meinsma}
\address{
School of Mathematics and Statistics, University of Sheffield,
Hounsfield Road, S3 7RH, UK, and
Hausdorff Center for Mathematics,
Universit\"at Bonn, Endenicher Allee 60, 53115 Bonn, Germany.}
\email{r.r.meinsma@gmail.com}

\begin{document}
\begin{abstract}
    We investigate obstruction classes of moduli spaces of sheaves on K3 surfaces. We extend previous results by C\u ald\u araru, explicitly determining the obstruction class and its order in the Brauer group. Our main theorem establishes a short exact sequence relating the Brauer group of the moduli space to that of the underlying K3 surface. This provides a criterion for when the moduli space is fine, generalising well-known results for K3 surfaces. Additionally, we explore applications to Ogg--Shafarevich theory for Beauville--Mukai systems. Furthermore, we investigate birational equivalences of Beauville–Mukai systems on elliptic K3 surfaces, presenting a complete characterisation of such equivalences.
\end{abstract}
\maketitle
\tableofcontents

\section{Introduction}
For a K3 surface $S$ with primitive Mukai vector $v$ and polarisation $H$, one can construct the moduli space $M_H(v)$ of $H$-Gieseker semistable sheaves with Mukai vector $v$. It is a hyperk\"ahler manifold of dimension $v^2+2$, which is deformation equivalent to a Hilbert scheme of points on a K3 surface, provided $H$ is generic \cite{GH96,OGr97,Huy06,BM14ii,BM14i}. These moduli spaces provide some of the most important examples of hyperk\"ahler manifolds, and their geometry is a subject of great interest. 

The moduli space $M_H(v)$ is generally a coarse moduli space. That is, there is not always a universal sheaf $\calE$ on $S\times M_H(v)$ for which $\calE|_{S\times [\calF]}\simeq \calF$ for all $[\calF]\in M_H(v)$. The obstruction to the existence of a universal sheaf is a Brauer class on $M_H(v)$. 
More precisely, there is a unique Brauer class $\alpha\in \Br(M_H(v))$ for which there exists a $(1\boxtimes \alpha)$-twisted universal sheaf on $S\times M_H(v)$. This Brauer class is called the \textit{obstruction class} of $M_H(v)$. If we assume $v^2=0$, then $M_H(v)$ is a K3 surface, and in this setting the obstruction class was computed explicitly in terms of cohomology by C\u ald\u araru in his PhD thesis \cite{Cal00}. His results showed that there exists a short exact sequence 
\begin{equation}\label{eq:introsesbrauer}
0\to \langle\alpha\rangle\to \Br(M_H(v))\to \Br(S)\to 0,
\end{equation}
and his very precise description of $\alpha$ also tells us the order of $\alpha$ in $\Br(M_H(v))$. The order of $\alpha$ equals the so-called \textit{divisibility} of the Mukai vector $v$, which is a number that is easy to compute in practice.

Before we explain our main result, let us briefly note some of the applications of C\u ald\u araru's work. Firstly, since a universal sheaf on $S\times M_H(v)$ exists if and only if the obstruction class $\alpha$ is trivial, it is a direct consequence of \eqref{eq:introsesbrauer} that $M_H(v)$ is fine if and only if the divisibility of $v$ is $1$. This is important partially because of its implications for the derived categories of $S$ and $M_H(v)$. Short exact sequence \eqref{eq:introsesbrauer} is obtained from the well-known short exact sequence \cite[Proposition 6.4]{Muk87} 
\begin{equation}\label{eq:introsestranscendental}
    0\to T(S)\to T(M_H(v))\to \Z/\dv(v)\Z\to 0,
\end{equation}
which shows that $T(S)$ is Hodge isometric to $T(M_H(v))$ if and only if $M_H(v)$ is a fine moduli space. Recall the Derived Torelli Theorem for K3 surfaces:
\begin{theorem}[Derived Torelli Theorem]\cite{Muk87,Orl03} \label{thm:introderivedtorelli}
    Let $X$ and $Y$ be two K3 surfaces. The following are equivalent:
    \begin{enumerate}
        \item[i)] There is an equivalence $\Db(X)\simeq \Db(Y)$;
        \item[ii)] There is a Hodge isometry $T(X)\simeq T(Y)$;
        \item[iii)] The K3 surface $Y$ is isomorphic to a fine moduli space of sheaves on $X$.
    \end{enumerate}
\end{theorem}

The main goal of this paper is to extend C\u ald\u araru's results about Brauer groups to higher-dimensional moduli spaces of sheaves. Our main result is the following theorem. For a K3 surface $S$, we denote by $\N(S) \coloneqq H^0(S,\Z)\oplus \NS(S) \oplus H^4(S,\Z)$ the \textit{extended N\'eron--Severi lattice}.

\begin{theorem}[See Theorem \ref{thm:explicitobstruction}]\label{thm:introexplicitobstruction}
    Let $S$ be a K3 surface, and let $v\in \N(S)$ be a primitive Mukai vector of square $v^2>0$. Let $H\in \NS(S)$ be a generic polarisation on $S$, and write $M\coloneqq M_H(v)$. Then there is a short exact sequence $$0\to \langle\alpha\rangle\to \Br(M)\to \Br(S)\to 0,$$ where $\alpha$ is the obstruction class for $M$ to be a fine moduli space of sheaves on $S$. Moreover, the order of $\alpha$ in $\Br(M)$ is equal to the divisibility of $v$ in $\N(S)$. 
\end{theorem}

Theorem \ref{thm:introexplicitobstruction} should be compared to \cite[Proposition 6.6]{KK24}, where the obstruction class is computed for four-dimensional moduli spaces of twisted sheaves using a different method which only applies in dimension four.

Theorem \ref{thm:introexplicitobstruction} is completely analogous to \eqref{eq:introsesbrauer}. In particular, it also allows us to easily determine whether a given moduli space is fine or not.

Let us say a few words on the proof of Theorem \ref{thm:introexplicitobstruction}. The proof is very closely related to C\u ald\u araru's proof in the case $v^2=0$. Some modifications are needed in the higher-dimensional setting. Most notably, short exact sequence \eqref{eq:introsestranscendental} has no analogue in the higher-dimensional setting, as in this case there is always a Hodge isometry $T(S)\simeq T(M)$. The way to derive \eqref{eq:introsesbrauer} from \eqref{eq:introsestranscendental} is by using the fact that $\Br(S)\simeq \operatorname{Hom}(T(S),\Q/\Z) \simeq T(S)^*\otimes \Q/\Z$ for a K3 surface $S$. However, this is not the case for higher-dimensional moduli spaces of sheaves on $S$. In this case, there is a certain subgroup $T'(M)\subset T(S)^*$, introduced in \cite{HM23}, for which there exists an isomorphism $\Br(M)\simeq T'(M)\otimes \Q/\Z$. The quotient $T(S)^*/T'(M)$ is cyclic of order $\dv(v)$, thus there is a short exact sequence $$0\to T'(M)\to T(S)^*\to \Z/\dv(v)\Z\to 0,$$ which gives rise to the short exact sequence of Theorem \ref{thm:introexplicitobstruction}.

The proof that the kernel of the short exact sequence of Theorem \ref{thm:introexplicitobstruction} is generated by the obstruction class goes via a deformation argument very similar to C\u ald\u araru's argument in the two-dimensional case. If $M$ is a fine two-dimensional moduli space of sheaves on $S$, then Mukai's and Orlov's work shows that the Fourier--Mukai equivalence associated to the universal sheaf on $S\times M$ is a Hodge isometry $\tilde{H}(S,\Z)\simeq \tilde{H}(M,\Z)$, where $\tilde{H}(S,\Z)$ denotes the Mukai lattice. If, on the other hand, $M$ has dimension $\dim M>2$, then such a statement does not hold, but we could achieve the same result using the \textit{Mukai pairing} on $H^*(M,\Q)$, which was introduced by C\u ald\u araru in \cite{Cal03}, which was later superseded by C\u ald\u araru and Willerton in \cite{CW10}.

\subsection*{Applications} 
Let us now explain some applications of Theorem \ref{thm:introexplicitobstruction}.
Firstly, one can apply \eqref{eq:introsesbrauer} in Ogg--Shafarevich theory for elliptic K3 surfaces. For an elliptic K3 surface with a section $S\to \PP^1$, one can consider the Tate--Shafarevich group $\Sha(S)$ consisting of elliptic K3 surfaces $X\to \PP^1$ with a fixed isomorphism between the Jacobian fibration $\Jac^0(X)\to \PP^1$ and $S\to \PP^1$. Such elliptic K3 surfaces are called \textit{torsors} of $S\to \PP^1$. The Tate--Shafarevich group of an elliptic K3 surface with a section is naturally isomorphic to its Brauer group, with the isomorphism being given by sending a torsor $X\to \PP^1$ to the obstruction to the existence of a universal sheaf on the product $X\times S$ by \cite[Theorem 4.4.1]{Cal00}. This description of the isomorphism 
\begin{equation}\label{eq:introoggshafarevich}
    \Sha(S)\simeq \Br(S)
\end{equation}
was used for example in \cite{MS24} to study derived equivalences between elliptic K3 surfaces.

Tate--Shafarevich groups were recently generalised to the higher-dimensional case \cite{Mar14,AR23,HM23}. For our applications, we rely on the new theory of Tate--Shafarevich groups introduced in \cite{HM23}. For $S$ a K3 surface, and $\calC\to |H|$ a complete, generically smooth linear system with $H^2 = 2g-2 \geq 0$, we may consider the relative Picard varieties $\olPic^d(\calC/|H|)\coloneqq M_L(0,H,d+1-g)$, where $d\in \Z$, and $L\in \NS(S)$ is a generic polarisation. These relative Picard varieties are called \textit{Beauville--Mukai systems}, and we usually denote them by $\olPic^d$ to keep notation light. For $d=0$, the generic fibre $\olPic^0_\eta$ is an abelian variety, and the Tate--Shafarevich group $\Sha(S,H)$, introduced in \cite{HM23}, parametrises those torsors of $\olPic^0_\eta$ which compactify to moduli spaces of (twisted) sheaves on $S$.

We prove that, for any $d\in \Z$, there is a short exact sequence $$0\to \Z/n\Z\to \Sha(S,H)\to \Br(\olPic^d)\to 0,$$ where $$n = \frac{\dv(H)}{\gcd(\dv(H),d+1-g)}.$$ Somewhat surprisingly, the morphism $\Sha(S,H)\to \Br(\olPic^0)$ is generally not an isomorphism, but instead has a cyclic kernel of order $2$ if the Picard rank of $S$ is 1. This is surprising because it contrasts with the isomorphism \eqref{eq:introoggshafarevich} for elliptic K3 surfaces. On the other hand, we have an isomorphism $\Sha(S,H)\simeq \Br(\olPic^{g-1})$. This isomorphism may be seen as a higher-dimensional analogue of \eqref{eq:introoggshafarevich}. 
Indeed, for an elliptic K3 surface $S\to \PP^1$ with a section and with fibre class $F\in \NS(S)$, we have $\Sha(S,F)\simeq \Sha(S)$ by \cite[\S5.1]{HM23}, and $\olPic^{g-1} = \olPic^0 = \Jac^0(S)\simeq S$,  see \cite[\S11.4]{Huy16} for details. 
%In this case, we have $F^2=2g-2=0$, hence $g-1=0$. This means that we have an equality $\olPic^{g-1} = \olPic^0$, and $\olPic^0 = \Jac^0(S)$ by definition. Moreover, since $S$ is an elliptic K3 with a section, we have $\Jac^0(S)\simeq S$, see \cite[\S11.4]{Huy16} for details. 
%Thus the isomorphism $\Sha(S,F)\simeq \Br(\olPic^{g-1})$ is another incarnation of \eqref{eq:introoggshafarevich}.

Moreover, using our explicit computation of the obstruction class of Theorem \ref{thm:introexplicitobstruction}, we show that for any $d\in \Z$, the morphism $\Sha(S,H)\to \Br(\olPic^d)$ maps the class $[\olPic^1]$ to the obstruction $\alpha_d\in \Br(\olPic^d)$ to the existence of a universal sheaf on $S\times \olPic^d$. Using this fact, we explain that a result by Addington, Donovan and Meachan \cite{ADM16} should be seen as an analogue to a result by Donagi and Pantev \cite{DP08}, see Theorem \ref{thm:ADM} and Theorem \ref{thm:donagi--pantev}, respectively.

Lastly, we study birational equivalence of moduli spaces of sheaves on elliptic K3 surfaces, for which we prove the following theorem.

\begin{theorem}[See Theorem \ref{thm:noderivedtorelli}] \label{thm:intronoderivedtorelli}
    Let $S$ be an elliptic K3 surface with a section, and let $M$ be a moduli space of sheaves on $S$ of dimension $2n$. Consider the following three statements.
    \begin{itemize}
        \item[i)] $M$ is a fine moduli space.
        \item[ii)] $M$ is birational to the Hilbert scheme $S^{[n]}$.
        \item[iii)] $M$ is derived equivalent to the Hilbert scheme $S^{[n]}$.
    \end{itemize}
    Then we have $$i) \iff ii) \implies iii).$$ 
\end{theorem}

The implication $iii)\implies ii)$ in Theorem \ref{thm:noderivedtorelli} was claimed in \cite{Bec22}, but the proof contains a mistake and does not hold in full generality. For instance, Zhang proves in \cite[Corollary 1.5]{Ruxuan} that all moduli spaces of dimension $4$, whether fine or not, on a fixed K3 surface are derived equivalent. In dimension greater than $4$, we have the following result, which gives a counterexample to a question raised in \cite[Problem 1.1]{KK24}, and a similar question raised in \cite[Question 2]{PR23}.

\begin{theorem}[See Theorem \ref{thm:coexampleDerTorelli}]
    For every $n\geq 3$, there exist moduli spaces of sheaves $M$ and $M'$ of dimension $2n$, on the same K3 surface, for which there is a Hodge isometry $T(M)\simeq T(M')$, but which are not derived equivalent.
\end{theorem}

We also apply our results on birational equivalence for moduli spaces of sheaves to Beauville--Mukai systems. Our main result in this direction is the following theorem, which fully describes birational equivalences of Beauville--Mukai systems in this setting.
\begin{theorem}[See Theorem \ref{thm:birationalequivalencebeauville--mukaioverelliptick3surfaces}]
    Let $S$ be an elliptic K3 surface with a section. Let $\calC\to |H|$ be a generically smooth complete linear system on $S$, and let $H^2 = 2g-2>0$. Assume $H$ is primitive in $\NS(S)$. The following are equivalent:
    \begin{itemize}
        \item[i)] $\olPic^d$ is birational to $\olPic^e$.
        \item[ii)] The obstruction classes of $\olPic^d$ and $\olPic^e$ have the same order. 
        \item[iii)] The Mukai vectors of $\olPic^d$ and $\olPic^e$ have the same divisibility. More precisely, we have $\gcd(\dv(H),d+1-g) = \gcd(\dv(H),e+1-g)$.
    \end{itemize}    
\end{theorem}

The implication $i)\implies ii)$ is a general fact about birational moduli spaces \cite[Proof of Lemma 9.10]{Bec22}. The equivalence $ii)\iff iii)$ is an easy consequence of Theorem \ref{thm:introexplicitobstruction}, and the implication $ii)\implies i)$ uses our explicit computation of the obstruction class.

\subsection*{Acknowledgements}
We warmly thank Mauro Varesco, Giacomo Mezzedimi, Daniel Huybrechts, and Lenny Taelman for helpful discussions and interest in our work. We are also thankful to Evgeny Shinder for many useful comments on an earlier draft of this paper. We would also like to express our gratitude to Nick Addington for incredibly helpful discussions about Lagrangian fibrations. We thank Ruxuan Zhang for pointing out a mistake in the first version of the paper and for many fruitful conversations. The second author wishes to thank the Hausdorff Center for Mathematics in Bonn for their hospitality and inspiring academic environment during his stay in the final two years of his PhD.

\section{Lattices and Brauer groups}

\subsection{Lattices}
Our primary source for lattice theory is \cite{Nik80}. A lattice is a free, finitely generated abelian group $L$ paired with a symmetric non-degenerate bilinear form $b: L\times L\rightarrow \Z$. We define the corresponding quadratic form $q(x) = b(x,x)$. We always denote $x \cdot y$ for $b(x,y)$ and $x^2$ for $q(x)$.  We call a lattice $L$ even if $x^2$ is even for every $x\in L$. For our considerations, all lattices are presumed to be even. 

A morphism of lattices between $(L,b)$ and $(L',b')$ is a group homomorphism $f: L\rightarrow L'$ that preserves the bilinear forms, i.e., that satisfies $b(x,y) = b'(f(x),f(y))$ for all $x,y\in L$. A bijective morphism between lattices is called an isometry. The group of isometries $L\simeq L$ is denoted $O(L)$.

A lattice embedding $N\hookrightarrow L$ is a primitive embedding if the quotient $L/N$ is torsion-free. Additionally, a vector $v\in L$ is primitive if $\langle v\rangle\hookrightarrow L$ is a primitive embedding.

\begin{definition}
    For any vector $v\in L$, the number $$\dv(v)\coloneqq \operatorname{gcd}_{u\in L}\left(u\cdot v\right)$$ is called the \textit{divisibility} of $v$.
\end{definition}

The dual of a lattice $L$ is the group $L^*\coloneqq\operatorname{Hom}(L,\Z)$. We can characterise the dual as a subgroup of $L\otimes \Q$ as follows:
\begin{equation}\label{eq:characterisingduallattice}
L^* \simeq \left\{x\in L\otimes \Q\mid \forall y\in L: x\cdot y\in \Z \right\} \subseteq L\otimes \Q.
\end{equation}
Via this characterisation, $L^*$ acquires a bilinear form taking values in $\Q$. 
There is a natural embedding $L\hookrightarrow L^*$ given by $$\ell\mapsto (\ell\cdot-).$$
The quotient $L^*/L = A_L$ is called the discriminant group of $L$. 
If the discriminant group is trivial, we call $L$ unimodular.
There is a natural quadratic form 
$\overline{q}: A_L\rightarrow \Q/2\Z$ induced by the quadratic form on $L^*$. The pair $(A_{L},\overline{q})$ is usually called the discriminant lattice (even though it is a finite group).
Any isometry $f: L\simeq L$ induces an isometry of rational quadratic forms $L\otimes \Q\simeq L\otimes \Q$, which induces an isometry of the discriminant lattice $\overline{f}:A_{L}\simeq A_{L}$. This gives a natural group homomorphism $$O(L)\to O(A_L).$$
Note that there is a natural embedding \begin{equation}\label{eq:discriminantembedding}
i_L:A_L\hookrightarrow L\otimes \Q/\Z,
\end{equation} induced by \eqref{eq:characterisingduallattice}. If we denote by $f_{\Q/\Z}:L\otimes\Q/\Z\simeq L\otimes \Q/\Z$ the isomorphism induced by $f\in O(L)$, then $$i_L\circ \overline{f} = f_{\Q/\Z}|_{i_L(A_L)}.$$ For any primitive sublattice $N\hookrightarrow L$, we denote $$i_{N,L}:A_N\hookrightarrow N\otimes \Q/\Z\hookrightarrow L\otimes \Q/\Z.$$ 

\begin{lemma}\label{lem:intersectingspecialbrauer}
    Let $L$ be a unimodular lattice, and let $N\hookrightarrow L$ be a primitive sublattice. Write $T\coloneqq N^\perp\subset L$ for the orthogonal complement. Then, in $L\otimes \Q/\Z$, we have $$(T\otimes \Q/\Z)\cap (N\otimes \Q/\Z) = i_{T,L}(A_{T}) = i_{N,L}(A_{N}).$$
\end{lemma}
\begin{proof}
    Suppose $x\in (T\otimes \Q/\Z)\cap (N\otimes \Q/\Z)$. Then there exist $\lambda\in T\otimes \Q$ and $v\in N\otimes \Q$ such that $x\equiv \lambda \pmod{L}$ and $x\equiv v\pmod{L}.$ In particular, we have $\lambda-v\equiv 0\pmod{L}.$ For any integral vector $\zeta\in T$, we have $\lambda\cdot \zeta = (\lambda-v)\cdot \zeta \in \Z$, so that $\lambda\in T^\vee$. This means that $x\in i_{T,L}(A_{T})$. By a similar argument, we have $x\in i_{N,L}(A_{N})$. This shows that $(T\otimes \Q/\Z)\cap (N\otimes \Q/\Z)$ is contained in $i_{T,L}(A_{T})$ and $i_{N,L}(A_{N}).$ 
    
    The other two inclusions follow from a standard argument, c.f. \cite[Proposition 6.4]{Muk87}. We include it here for completeness. Let $\lambda\in i_{T,L}(A_{T})$. Since it is clear that $\lambda\in T\otimes \Q/\Z$, we must show that $\lambda\in N\otimes \Q/\Z$. By slight abuse of notation, we also denote by $\lambda$ any lift to $T^\vee$. Since $T$ is a primitive sublattice of $L$ and $L$ is unimodular, there is a surjective map $L\to T^\vee$. That is, there exists a vector $x\in L$ such that $x\cdot\zeta = \lambda\cdot\zeta$ for every $\zeta\in T$. In particular, $x-\lambda$ is orthogonal to $T$, and thus $v\coloneqq x-\lambda\in N\otimes \Q$. Therefore we have $\lambda+v\in L$, hence $\lambda\equiv -v\pmod {L}$, and we obtain $\lambda\in N\otimes \Q/\Z$. The final remaining inclusion is completely analogous.
\end{proof}

\begin{remark}\label{rmk:discriminantorthogonal}
    Lemma \ref{lem:intersectingspecialbrauer} induces an isomorphism of groups $A_T\simeq i_{T,L}(A_T) = i_{N,L}(A_N)\simeq A_N$, which is an isomorphism of quadratic forms $A_{T}(-1)\simeq A_N$ \cite[Proposition 1.6.1]{Nik80}.
\end{remark}

\begin{lemma}\label{lem:nikulinorthogonal} \cite[Proposition 1.6.1]{Nik80}
    Let $L$ be a unimodular lattice, let $N\hookrightarrow L$ be a primitive sublattice and write $T\coloneqq N^\perp \subset L$. Suppose $g\in O(N)$ and $h\in O(T)$. Then $\overline{g}$ and $\overline{h}$ induce the same isometry on $i_{N,L}(A_N) = i_{T,L}(A_T)$ if and only if there is an isometry $f: L\simeq L$ preserving $N$ and $T$, that satisfies $f|_N = g$ and $f|_T = h$.
\end{lemma}

One lattice which is of particular importance is the so-called \textit{hyperbolic lattice} $U$, which is the rank 2 lattice given by the symmetric bilinear form $$\left(\begin{matrix}
0 & 1\\
1& 0
\end{matrix}\right).$$ This the even unimodular lattice with the lowest rank, which makes it very versatile, see for example the following two lemmas. 

\begin{lemma}\cite[Proposition 1.14.1]{Nik80}\label{lem:extendisometryhyperbolic}
    Let $L$ be a unimodular lattice, and let $N\hookrightarrow L$ be a primitive sublattice. If there exists an embedding $U\subset N^\perp$, then for any $g\in O(N)$, there exists an isometry $f: L\simeq L$ such that $f|_{N} = g$.
\end{lemma}
\begin{proof}
    Write $T = N^\perp$. By assumption, there is an orthogonal decomposition $T = U\oplus T'$ for some lattice $T'$. We have $A_T \simeq A_U\oplus A_{T'}\simeq A_{T'},$ since $U$ is unimodular. In particular, $\rk(T) = \rk(T')+2 \geq \ell(A_{T'})+2.$ Combined with \cite[Theorem 1.14.2]{Nik80}, this implies that the homomorphism $O(T)\to O(A_T)$ is surjective. Choose $h\in O(T)$ such that $\overline{h}$ induces the same isomorphism on $i_{T,L}(A_T)$ as $g$, then by Lemma \ref{lem:nikulinorthogonal} there exists an isometry $L\simeq L$ which restricts to $g$ on $N$.
\end{proof}

\begin{lemma}\label{lem:eichlertransvections}\cite[Proposition 3.3]{GHS07}
    Let $L = U\oplus U\oplus L'$ be a lattice containing two copies of the hyperbolic plane. Let $v,u\in L$ be two primitive vectors such that 
    \begin{enumerate}
        \item[i)] $\frac{v}{\dv(v)} = \frac{u}{\dv(u)}\in A_L$,
        \item[ii)] $u^2 = v^2.$
    \end{enumerate}
    Then there is an isometry $f: L\simeq L$ such that $\overline{f} = \operatorname{id}_{A_L}$ and such that $f(v) = u.$
\end{lemma}

If $N\hookrightarrow L$ is an embedding of lattices of finite index, then we say that $L$ is an overlattice of $N$. In this case, there is a sequence of embeddings $N\hookrightarrow L\hookrightarrow L^*\hookrightarrow N^*$, which exhibits $L/N$ as a subgroup of $A_N$, which is isotropic for the quadratic form on $A_N$.

\subsection{Lattices of hyperk\"ahler manifolds}
If $S$ is a complex projective K3 surface, the cup product turns $H^2(S,\Z)$ into an even, unimodular lattice isometric to $\Lambda_{\text{K3}} = U^{\oplus 3}\oplus E_{8}(-1)^{\oplus2}$, where $U$ is the hyperbolic lattice, and $E_8$ is the unique even, unimodular, positive-definite lattice of rank 8, see for example \cite[\S VIII.1]{BPV12} for more details.

The lattice $H^2(S,\Z)$ carries a natural Hodge structure, which is of K3 type, meaning that it is of weight 2 and $H^{2,0}(S) \simeq \CC\sigma$. We call an integral Hodge structure which admits a lattice structure a \textit{Hodge lattice}.

The integral $(1,1)$-part of $H^2(S,\Z)$ is called the N\'eron--Severi lattice $\NS(S)\coloneqq H^{1,1}(S,\Z)$.
The orthogonal complement $T(S) = \operatorname{NS}(S)^\perp \subseteq H^2(S,\Z)$ is called the transcendental lattice. Equivalently, $T(S)$ is the smallest primitive sublattice of $H^2(S,\Z)$ such that $T(S)_\CC$ contains $H^{2,0}(S)$, see for example \cite[\S3.2.2]{Huy16}.

The full integral cohomology group $H^*(S,\Z) = H^0(S,\Z)\oplus H^2(S,\Z)\oplus H^4(S,\Z)$, admits a lattice structure given by $$(r,l,s)\cdot(r',l',s') = l\cdot l'-rs'-sr'.$$ It also inherits a Hodge from $H^2(S,\Z)$ and we call the resulting Hodge lattice the \textit{Mukai lattice}, and denote it by $\tilde{H}(S,\Z)$. We define the \textit{extended N\'eron-Severi lattice} of $S$ as the $(1,1)$-part of $\tilde{H}(S,\Z)$: $$\N(S)\coloneqq \tilde{H}^{1,1}(S,\Z) = H^0(S,\Z)\oplus \NS(S)\oplus H^4(S,\Z).$$ As a lattice, we have  $\N(S)\simeq \NS(S)\oplus U.$

\begin{definition}\cite{MS24}\label{def:caldararuclass}
For a vector $v\in \N(S)$, we call the element $a_v\coloneqq \frac{v}{\dv(v)}\in A_{\N(S)}$ the \textit{C\u ald\u araru class of $v$.} The element $\omega_v\in A_{T(S)}$ corresponding to $a_v$ via the natural isomorphism $A_{T(S)}(-1)\simeq A_{\N(S)}$ is called the \textit{transcendental C\u ald\u araru class of $v$.}
\end{definition}

If $X$ is a complex projective hyperk\"ahler manifold, the second cohomology group $H^2(X,\Z)$ also carries a lattice structure, given by the \textit{Beauville--Bogomolov--Fujiki} (\textit{BBF}) form $q$ \cite{Bea83,Bog96,Fuj87}. It is important to note that, unlike in the case for K3 surfaces, $H^2(X,\Z)$ is not unimodular if $\operatorname{dim}(X)>2$ in all the currently known deformation types \cite{Rap08}. 

Similarly to the case of K3 surfaces, we write $\NS(X)\subset H^2(X,\Z)$ for $H^{1,1}(X,\Z)$, and $T(X)\coloneqq \NS(X)^\perp$.

We are mostly interested in hyperk\"ahler manifolds of \Kn-type, that is, hyperk\"ahler manifolds that are deformation equivalent to the Hilbert scheme of $n$ points on a K3 surface. Assuming $n\geq 2$, we have $$H^2(X,\Z)\simeq U^{\oplus 3}\oplus E_8(-1)^{\oplus 2}\oplus \langle-2n-2\rangle.$$

\subsection{Brauer groups of hyperk\"ahler manifolds}
For basic facts about Brauer groups, we refer to \cite{CS21}.
\begin{definition}
    An Azumaya algebra $\calA$ on a scheme $X$ is a sheaf of $\bigo_X$-algebras, which is \'etale locally isomorphic to the sheaf of matrix algebras $\operatorname{M}_{n\times n}(\bigo_X)$. 
\end{definition}

We say that two Azumaya algebras $\calA$, $\calB$ are \textit{$\Br$-equivalent} if there exist locally free sheaves $\calE$, $\calF$ on $X$ such that 
\begin{equation}\label{eq:brequiv}\calA\otimes \SEnd(\calE)\simeq \calB\otimes \SEnd(\calF).
\end{equation}

The set of isomorphism classes of Azumaya algebras carries a natural group structure given by the tensor product, and $\Br$-equivalence is compatible with this group structure \cite[Theorem 5.1]{Gro68}.

\begin{definition}
    The Brauer group $\Br(X)$ of $X$ is the group of $\Br$-equivalence classes of Azumaya algebras of $X$.
\end{definition}

Now, let $X$ be a complex smooth projective variety.

\begin{definition}
    The \textit{cohomological Brauer Group} of $X$ is the group $H^2_{\et}(X,\mathbb{G}_m)$, or, equivalently, $H^2(X,\bigo_X^\times)_{\text{tors}}.$ 
\end{definition}
The fact that $H^2_{\et}(X,\bbG_m)$ and $H^2(X,\bigo_X^\times)_{\text{tors}}$ are isomorphic is a consequence of the fact that $H^2(X,\mathbb{G}_m)$ is torsion, see \cite[Remark 18.1.4ii)]{Huy16}, combined with the Kummer sequence, see \cite[Remark 11.5.13]{Huy16}.

It is a result by De Jong and Gabber that the Brauer group is naturally isomorphic to the cohomological Brauer group, see \cite[\S4.2]{CS21}.

Under the additional assumption that $H^3(X,\Z)=0$, it follows from the exponential sequence that there is an isomorphism 

\begin{equation}\label{eq:brauergroup}
    \Br(X)\simeq \frac{H^2(X,\Z)}{\NS(X)}\otimes \Q/\Z.
\end{equation}
For this reason, as in \cite{HM23}, we write 
\begin{equation}\label{eq:Tprime}
    T'(X) \coloneqq \frac{H^2(X,\Z)}{\NS(X)}.
\end{equation}

We note that the additional assumption that $H^3(X,\Z) = 0$ is satisfied for hyperk\"ahler manifolds of \Kn-type. This follows from \cite[Equation (2.1)]{Got02} and \cite[Theorem 1]{Mar07}.

There is a well-defined injective group homomorphism $T'(X)\to T(X)^*$ given by 
\begin{equation}\label{eq:morphismtranscendentals}[\xi]\mapsto (\xi\cdot-)|_{T(X)}.
\end{equation}
Since $T(X)\subset H^2(X,\Z)$ is a primitive sublattice, the restriction map $H^2(X,\Z)^*\to T(X)^*$ is surjective. This means that if $H^2(X,\Z)$ is unimodular, i.e. if $H^2(X,\Z)\simeq H^2(X,\Z)^*$, then for any element $f\in T(X)^*$, there exists a vector $\xi\in H^2(X,\Z)$ such that $f(t) = \xi\cdot t$ for all $t\in T(X)$. In this case, the morphism \eqref{eq:morphismtranscendentals} is an isomorphism.
Recall that $H^2(S,\Z)$ is unimodular if $S$ is a K3 surface, hence the above discussion proves the following well-known lemma.
\begin{lemma}\label{lem:brauerofak3}
    Let $S$ be a K3 surface. Then there is an isomorphism $$\Br(S)\simeq T(S)^*\otimes \Q/\Z\simeq \operatorname{Hom}(T(S),\Q/\Z).$$ 
\end{lemma}
To summarise, for a K3 surface $S$, there are natural isomorphisms 
\begin{equation}\label{eq:allthebrauergroups}
\Br(S)\simeq H^2_{\text{\'et}}(S,\mathbb{G}_m)\simeq H^2(S,\bigo_S^\times)_{\text{tors}}\simeq T'(S)\otimes\Q/\Z\simeq\operatorname{Hom}(T(S),\Q/\Z),
\end{equation}
and we refer to each of these groups as the Brauer group of $S$.

On the other hand, if $X$ is a hyperk\"ahler manifold of \Kn-type, for some $n\geq 2$, then $H^2(X,\Z)$ is never unimodular, and generally we have $T'(X)\not\simeq T(X)^*$, and hence the Brauer group of $X$, which is isomorphic to $T'(X)\otimes \Q/\Z$ by \eqref{eq:brauergroup}, is generally not isomorphic to $\Hom(T(X),\Q/\Z)$.

\section{Moduli spaces}
\subsection{Moduli spaces of sheaves on K3 surfaces}
In the following, we denote $v(E)=\ch(E)\sqrt{\td_X}$ for a sheaf $E\in\Coh(X)$ on a projective variety $X$. The following theorem, based on the pioneering work of Mukai \cite{Muk87}, is well known and due to many authors \cite{GH96,Huy06,OGr97}. A more general version can be found in \cite{BM14i, BM14ii}.

\begin{theorem}\label{thm: moduli sheaves K3}
    Let $S$ be a K3 surface and $v\in \widetilde{H}(S,\matZ)$ a primitive Mukai vector. For $H$ a generic polarisation, there exists a (possibly empty) coarse moduli space $M_{H}(v)$ of $H$-Gieseker stable coherent sheaves $E$ with $v(E)=v$. Moreover:
    \begin{enumerate}
        \item $M_H(v)$ is empty if $v^2<-2$,
        \item if $v=(r,D,s)$ satisfies $v^2\geq -2$ and either
        \begin{enumerate}
            \item $r>0$,
            \item $r=0$, $D$ effective,
            \item $r=D=0$, $s>0$,
        \end{enumerate}
        then $M_H(v)$ is a projective hyperk\"ahler manifold of $K3^{[n]}$-type, where $2n= v^2+2$.
    \end{enumerate}
\end{theorem}

Given $v$ as in Theorem \ref{thm: moduli sheaves K3}${\rm (2)}$, a polarisation $H$ such that $M_H(v)$ is smooth is called $v$-generic. When no confusion is possible, we write $M(v)$ instead of $M_H(v)$. It turns out that two choices $H,H'$ of $v$-generic polarisation give rise to two birationally equivalent moduli spaces $M_H(v)$ and $M_{H'}(v)$ \cite[Theorem 1.1]{BM14ii}.

Let $S$ be a K3 surface and fix a primitive Mukai vector $v\in \N(S)$. For a $v$-generic polarisation $H$, we write $M = M_H(v)$. Recall that $M$ is a priori only a \textit{coarse} moduli space, meaning that there is not necessarily a universal sheaf on $S\times M$. However, twisted universal sheaves and quasi-universal sheaves always exist, and we recall some basic facts about them now. Our basic reference for twisted sheaves is \cite{Cal00}.

\begin{definition}
    Let $\alpha_M\in \Br(M)$ be some Brauer class. A $(1\boxtimes\alpha_M)$-twisted universal sheaf is a $(1\boxtimes \alpha_M)$-twisted sheaf $\calU$ on $S\times M$ such that $\calU|_{S\times [\calE]}\simeq \calE$ for all $[\calE]\in M$. If $\alpha_M$ is trivial, we simply call $\calU$ a universal sheaf on $S\times M$. If a $(1\boxtimes\alpha_M)$-twisted universal sheaf exists, we call $\alpha_M$ the obstruction to the existence of a universal sheaf.
\end{definition}

We will see in Definition/Proposition \ref{defprop:universaltwistedsheaves} that, unlike universal sheaves, \textit{twisted} universal sheaves always exist. Moreover, the Brauer class $\alpha_M$ is unique, and therefore we call it the \textit{obstruction class}. We state this result in the relative setting, as we will need to be able to deal with families of moduli spaces in the next sections.

\begin{proposition}\cite{HL10}
\label{prop:relativetwistedmodulispaces}
    Let $f\colon  \calS\to T$ be a projective morphism of schemes of finite type over $\CC$ with connected fibres. Let $\bigo_{\calS}(1)$ be a relatively ample line bundle on $\calS$.  Then for any polynomial $P$ there exists a coarse relative moduli space $\calM\coloneqq \calM_{\calS/T}(P)\to T$ for the functor $$\catname{M}\colon  (\catname{Sch/T})^\circ\to \catname{Sets}$$
    which associates to a $T$-scheme $X\to T$ the set of isomorphism classes of $T$-flat families of stable sheaves on the fibres of $\calS\times_TX\to X$ with Hilbert polynomial $P$. In particular, for any $t\in T$, we have $$\calM_t\simeq M_{\calS_t}(P).$$ 
\end{proposition}

\begin{defprop}\cite[Proposition 3.3.2]{Cal00}
\label{defprop:universaltwistedsheaves}
    In the notation of Proposition \ref{prop:relativetwistedmodulispaces}, there exists a unique Brauer class $\alpha_\calM\in \Br(\calM)$ such that there is a $(1\boxtimes\alpha_\calM)$-twisted universal sheaf on $\calS\times_T\calM$. This Brauer class is called the obstruction to the existence of a universal sheaf on $\calS\times_T\calM$. For any point $t\in T$, this twisted universal sheaf restricts to a $(1\boxtimes \alpha_t)$-twisted universal sheaf on $\calS_t\times\calM_t$, hence $\alpha_t$ is the obstruction to the existence of a universal sheaf on $\calS_t\times\calM_t$.
\end{defprop}

The main goal of this paper is to compute the obstruction Brauer class of a moduli space of sheaves on a K3 surface. 
The first result in this direction is the following.

\begin{proposition}\cite[Theorem 4.6.5]{HL10} \label{prop:fineifdivone}
    Let $S$ be a K3 surface with primitive Mukai vector $v\in \N(S)$ and $v$-generic polarisation $H$. If $\dv(v) = 1$, then $M_H(v)$ is a fine moduli space.
\end{proposition}
Our main theorem implies the converse of Proposition \ref{prop:fineifdivone}, c.f. Corollary \ref{cor:fineiffdivone}.

Now, we fix $S,v,H$ such that $M\coloneqq M_H(v)$ is non-empty and smooth. Moreover, we assume $v^2\geq 2$ so that the dimension of $M$ is at least 4. We denote by $p$, resp. $q$ the projection from $S\times M$ to $M$, resp. to $S$. Let $\alpha_M$ be the obstruction to the existence of a universal sheaf on $S\times M$, and let  $\calE$ be a $(1\boxtimes \alpha_M)$-twisted universal sheaf. There exists an $\alpha_M^{-1}$-twisted locally free sheaf $\calF$ of finite rank $\rho$ on $M$ by \cite[Theorem 1.3.5]{Cal00}. Now $\calU\coloneqq \calE\otimes p^*\calF$ is a \textit{quasi-universal sheaf of similitude $\rho$} on $S\times M$. This means that $\calE\otimes p^*\calF$ is an $M$-flat untwisted sheaf on $S\times M$ with the property that 
for any $[F]\in M$, we have $\calU|_{S\times[F]}\simeq F^{\oplus \rho}$.  We consider the Fourier--Mukai transform
$$\fonction{\Phi^{\calU^\vee}}{\Db(S)}{\Db(M)}{F}{Rp_*(q^*F\otimes \calU^\vee)}.$$
This Fourier--Mukai transform depends on the choices of $\calE$ and $\calF$. However, if $\calU$ and $\calV$ are two quasi-universal sheaves on $S\times M$, then there exist vector bundles $E$ and $F$ on $M$ such that $\calU\otimes p^*E\simeq \calV\otimes p^*F$ by \cite[Appendix 2]{Muk87}. 

\begin{definition}\label{def:mukaimorphism}
    Let $S$ be a K3 surface with primitive Mukai vector $v\in \N(S)$ for which $v^2\geq 2$. Let $H$ be a $v$-generic polarisation, and write $M\coloneqq M_H(v)$. Let $\calU$ be a quasi-universal sheaf on $S\times M$ of similitude $\rho$. The normalised cohomological Fourier--Mukai transform
    $$\fonction{\varphi\coloneqq \frac{1}{\rho}\varphi^{v(\calU^\vee)}}{H^*(S,\matQ)}{H^*(M,\matQ)}{x}{\frac{1}{\rho}p_*\left(q^*x\otimes v(\calU^\vee)\right),}$$
    is called the \textit{Mukai morphism}.
\end{definition}

Note that the Mukai morphism is also dependent on the choice of quasi-universal sheaf. More precisely, let $F$ be a vector bundle on $M$ and write $\calU' = \calU\otimes p^*F$. Let $\varphi'$ be the Mukai morphism corresponding to $\calU'$, then $$\varphi'(x) = \frac{\ch(F)}{\rk(F)}\varphi(x).$$ In particular, if we only consider the degree 2 part, we find $$[\varphi'(x)]_2 = [\varphi(x)]_2+\frac{c_1(F)}{\rk(F)}[\varphi(x)]_0.$$ 

\begin{lemma} \cite[Proof of Theorem 1.5]{Muk87}, \cite{OGr97} \label{lem:mukaipairingmoduli}
    For any $x\in H^*(S,\Q)$, we have $$[\varphi(x)]_0 = -x\cdot v.$$ In particular, whenever $x\in v^\perp\subset H^*(S,\Q)$, $[\varphi(x)]_2$ is independent of the choice of quasi-universal sheaf.
\end{lemma}
\begin{proof}
    Let $\omega\in H^{2n}(M,\Z)$ be the fundamental cocycle. Let $\Phi^{\calU^\vee}:\Db(S)\to \Db(M)$ be the Fourier--Mukai transform with kernel $\calU^\vee$, and let $\varphi^{v(\calU^\vee)}: H^*(S,\Q)\to H^*(M,\Q)$ be the corresponding cohomological Fourier--Mukai transform, i.e. $\varphi^{v(\calU^\vee)} = \rho\cdot\varphi$.
    In the notation of \cite{Cal03},
    it follows that we have $$[\varphi^{v(\calU^\vee)}(x)]_0= (\varphi^{v(\calU^\vee)}(x),\omega) = (x,\varphi^{v(\calU^\vee)}_L(\omega)),$$
    where $\varphi^{v(\calU^\vee)}_L=\varphi^{v(\calU)}$ is the left-adjoint of $\varphi^{v(\calU^\vee)}$. On the other hand, we have $\varphi^{v(\calU^\vee)}_L(\omega) = v(\Phi^{\calU}(\calO_t))= \rho v$ for any $t\in M$. Hence we obtain $[\varphi(x)]_0 = \frac{1}{\rho}[\varphi^{v(\calU^\vee)}(x)]_0 = \frac{1}{\rho}(x,\rho v) = -x\cdot v$, as claimed.
\end{proof}

\begin{theorem}\cite{OGr97,Yos01}\label{thm:mukaimorphism}
    Let $S$ be a K3 surface, and let $v\in \N(S)$ be a primitive Mukai vector with $v^2>0$. Let $H$ be a $v$-generic polarisation, and denote $M\coloneqq M_H(v)$. Let $\varphi: H^*(S,\Q)\to H^*(M,\Q)$ be the Mukai morphism from Definition \ref{def:mukaimorphism}. 
    Then $\varphi$ induces a Hodge isometry
    \begin{equation*}
           \isomorphismstar{H^*(S,\Z)\supset v^\perp}{H^2(M,\Z)}{x}{[\varphi(x)]_2.} 
    \end{equation*}
    Here, $[\varphi(x)]_2$ denotes the degree 2 part of $\varphi(x)$. Moreover, this Hodge isometry is independent of the choice of a quasi-universal sheaf on $S\times M$.
\end{theorem}

When $v^2=0$, $M$ is a K3 surface, and the theorem holds if one replaces $v^\perp$ by $v^\perp/v$ \cite[Theorem 1.5]{Muk87}.

\subsection{The obstruction class of a two-dimensional moduli space}
In this subsection, let $S$ be a K3 surface with Mukai vector $v\in \N(S)$. Moreover, we assume that $v^2 = 0$. In this case, the moduli space $M\coloneqq M_H(v)$ is a K3 surface, where $H$ is a $v$-generic polarisation. In this setting, C\u ald\u araru computed the obstruction class $\alpha_M$ explicitly as a class in $\operatorname{Hom}(T(M),\Q/\Z) \simeq \Br(M)$. We now give a brief summary of his results.

Recall that the Mukai morphism is a Hodge isometry $v^\perp/v\simeq H^2(M,\Z)$. Since we have $T(S)\subset v^\perp\subset \tilde{H}(S,\Z)$, and $v\notin T(S)$, the Mukai morphism induces an embedding of Hodge lattices $T(S)\hookrightarrow T(M)$. Moreover, the quotient $T(M)/T(S)$ is finite and cyclic. 
There is a natural embedding $T(M)/T(S)\subset A_{T(S)}$. Via this embedding, $T(M)/T(S)$ is generated by the transcendental C\u ald\u araru class of $v$, denoted $\omega_v$ by \cite[Proposition 6.4]{Muk87} (see Definition \ref{def:caldararuclass}). This defines a Brauer class $\alpha_M$ on $M$ as the composition 
\begin{equation}\label{eq:obstructionbrauerclassfork3s}
    \alpha_M:T(M)\to T(M)/T(S)\simeq \Z/\dv(v)\Z,
\end{equation}
where the isomorphism $T(M)/T(S)\simeq \Z/\dv(v)\Z$ is the one that maps $\omega_v$ to $-\ol{1}\in \Z/\dv(v)\Z$. 

\begin{theorem}\cite{Cal00} \label{thm:obstructionsfork3s}
    The Brauer class $\alpha_M$ of equation \eqref{eq:obstructionbrauerclassfork3s} is the obstruction to the existence of a universal sheaf on $S\times M$. 
\end{theorem}

The description of the obstruction Brauer class $\alpha_M$ of Theorem \ref{thm:obstructionsfork3s} is useful because it is so explicit. For example, it is essential in \cite{MS24} to study Fourier--Mukai partners of elliptic K3 surfaces. Unfortunately, the description uses the fact that $\Br(M)\simeq \Hom(T(M),\Q/\Z)$, which is only the case when $M$ is a K3 surface. Another, equivalent, way to view $\alpha_M$ is as follows.

Consider the short exact sequence 
\begin{equation}\label{eq:sestranscendentalfork3s}
    0\to T(S) \to T(M) \to \langle \omega_v\rangle \to 0,
\end{equation}
which induces the following short exact sequence by taking duals:
\begin{equation}\label{eq:sestranscendentalduals}
    0 \to T(M)^* \to T(S)^* \to \Z/\dv(v)\Z \to 0.
\end{equation}

Since $v\in \tilde{H}(S,\Z)$ is primitive, and $\tilde{H}(S,\Z)$ is unimodular, it follows that the divisibility of $v$ in $\tilde{H}(S,\Z)$ is $1$. That is, there exists a vector $u\in \tilde{H}(S,\Z)$ with the property that $u\cdot v = 1$. 
We claim that $$(u\cdot -)|_{T(S)} \in T(S)^*,$$ which is the image of $u$ under the natural surjection $\tilde{H}(S,\Z) \to T(S)^*$, induces a generator of $T(S)^*/T(M)^*$. This is proven in Lemma \ref{lem:sestranscendental} in a more general setting. We abuse notation slightly and write $u$ for $(u\cdot -)|_{T(S)}$. Let us write $w\in T(M)^*$ for the element which maps to $\dv(v)\cdot u \in T(S)^*$.

Note that taking the tensor product of \eqref{eq:sestranscendentalduals} with $\Q/\Z$ yields:
\begin{equation}\label{eq:sesbrauerfork3}
    0 \to \Z/\dv(v)\Z \to \Br(M) \to \Br(S) \to 0.
\end{equation}
The kernel of \eqref{eq:sesbrauerfork3} is generated by $$\frac{w}{\dv(v)} \in \Br(M),$$ and this is precisely the obstruction class $\alpha_M$ of Theorem \ref{thm:obstructionsfork3s}, see \cite{Cal00} and \cite{MS24}. 

Note that the order of $\alpha_M$ in $\Br(M)$ is equal to the divisibility $\dv(v)$ of $v$ in $\N(S)$. We can see this from short exact sequence \eqref{eq:sesbrauerfork3}, since the kernel is generated by the obstruction class. We can also use \eqref{eq:obstructionbrauerclassfork3s} to find the order of $\alpha_M$, since \eqref{eq:obstructionbrauerclassfork3s} exhibits $\alpha_M$ as a surjective group homomorphism $\alpha_M\colon T(M)\to \Z/\dv(v)\Z$.

\begin{corollary} \cite{Cal00} \label{cor:fineiffdivoneindimensiontwo}
    Let $S$ be a K3 surface with primitive Mukai vector $v\in \N(S)$ and $v$-generic polarisation $H$. The moduli space $M_H(v)$ is fine if and only if $\dv(v)=1$.
\end{corollary}

\subsection{Moduli spaces of lattice-polarised K3 surfaces}
In this section, we collect some basic facts about moduli spaces of lattice polarised K3 surfaces. Our main reference is \cite{Dol95}. An important technical result in this section is Lemma \ref{lem:divisibilityinfamilies}, which is used to prove the main result of Section \ref{sec:identifyingobstructionclass} below. 

\begin{definition}
    For a lattice $T$ of signature $(2,n)$, we define the period domain $\Omega_T$ of $T$ to be one of the two connected components of 
    \begin{equation}\label{eq:nottheperioddomain}
    \left\{\sigma\in \PP(T\otimes \CC)\mid \sigma^2 = 0 \text{ and } \sigma\cdot\overline{\sigma}>0\right\}.
    \end{equation}
\end{definition}

The orthogonal group $O(T)$ acts naturally on the set \eqref{eq:nottheperioddomain}, and we write $$O^+(T)\coloneqq \left\{\sigma\in O(T)\mid \sigma(\Omega_T) = \Omega_T\right\}$$ for the subgroup of $O(T)$ consisting of isometries that preserve the connected component $\Omega_T$. The index of $O^+(T)\subset O(T)$ is two.

Let $N$ be an even lattice of signature $(1,\rho-1)$ for some $1\leq \rho \leq 20$. Suppose that there exists precisely one $O(\klat)$-orbit of primitive embeddings $N\hookrightarrow \klat$. This is the case for example when $\rho\leq 10$ by \cite[Theorem 1.14.4]{Nik80}.
An \textit{$N$-marked K3 surface} is a pair $(S,j)$ consisting of a K3 surface $S$ and a primitive embedding $j:N\hookrightarrow \NS(S)$. 
An isomorphism of $N$-marked K3 surfaces $(X,i)$, $(Y,j)$ is an isomorphism $f: X\simeq Y$ such that $i = f^*\circ j$.

A \textit{marked K3 surface} is a pair $(S,\phi)$ consisting of a K3 surface $S$ and an isometry $\phi: H^2(S,\Z)\simeq \klat$. Fix an embedding $N\subset \klat$. If $(S,\phi)$ is a marked K3 surface such that $N\subset \phi(\NS(S))$, then $(S,\phi^{-1}|_N)$ is an $N$-marked K3 surface.

Let $T = N^\perp\subset \klat$ be the orthogonal complement. We write $$\toplus(T)\coloneqq \ker\left(O^+(T)\to O(A_T)\right).$$ If $(S,j)$ is a $N$-marked K3 surface, then there exists a marking $\phi: H^2(S,\Z)\simeq \klat$ such that the period of the marked K3 surface $(S,\phi)$, i.e. $[\phi(H^{2,0}(S,\bbC))]$, lies in $\Omega_T$. Moreover, two isomorphic $N$-marked K3 surfaces give rise to periods which lie in the same $\toplus(T)$-orbit.

\begin{definition}
    We denote $$\calF_{T}\coloneqq \toplus(T)\setminus \Omega_T.$$
\end{definition}
\begin{theorem}\cite[\S3]{Dol95}
    The quotient $\calF_T$ is the coarse moduli space of $N$-marked K3 surfaces. The moduli space $\calF_T$ is an equidimensional quasi-projective variety of dimension $20-\rho$.
\end{theorem}

Let $S$ be a K3 surface, and let $v\in \N(S)$ be a primitive Mukai vector. Write $v = (r,E,s)$ for some $r,s\in \Z$ and $E\in \NS(S)$. Let $H\in \NS(S)$ be an ample divisor. 
Let $M$ be the saturation of $\langle H,E\rangle$ in $\NS(S)$. That is, $M=\left(\langle H,E\rangle\otimes \Q\right) \cap \NS(S)$ is the smallest primitive sublattice of $\NS(S)$ which contains $\langle H,E\rangle$. 
Fix any marking $\phi: H^2(S,\Z)\simeq \klat$. Write $L\coloneqq \phi(M)\subset \klat$ and $e = \phi(E)$, $h = \phi(H)$. A point in the moduli space $\calF_{L^\perp}$ corresponds to an $L$-marked K3 surface $(X,i)$ for which the vector $(r,i(e),s)\in \N(X)$ is a Mukai vector, which we denote by $i(v)$.
We wish to show that the locus in $\calF_{L^\perp}$ of $L$-marked K3 surfaces $(X,i)$ for which the Mukai vector $i(v)$ has divisibility $1$ is dense.  

From the definition of the period domain, it follows that we have $$\Omega_S = \Omega_T\cap \PP(S\otimes \CC),$$ for any primitive sublattice $S\subset T$. By a slight abuse of notation, we write $\calF_S\subset \calF_T$ for the image of $\Omega_S$ along the natural projection $\Omega_T\twoheadrightarrow \calF_T$. If we have $\rk(S)+1 = \rk(T)$, then $\calF_S$ is a divisor in $\calF_T$. 

\begin{remark}\label{rem:perioddomainknowsaboutthelattice}
    Fix a primitive sublattice $N \subset \klat$ of signature $(1,\rho-1)$ and write $S = N^\perp$. For a point $[\ell]\in \Omega_S$, we denote by $\NS_\ell$ the integral $(1,1)$-part of the Hodge lattice of K3 type whose underlying lattice is $\klat$ and whose $(2,0)$-part is $\ell$. It is a well-known fact that, for a very general point $[\ell]$, we have $\NS_\ell = N$. To see this, suppose that $[\ell]\in \Omega_N$ satisfies $N\not\subset\NS_\ell.$ Then it follows from the definitions that $[\ell]\in \Omega_{\NS_\ell^\perp}\subset \Omega_S$. Recall that we have $\dim(\Omega_{\NS_\ell^\perp}) = \rk(\NS_\ell^\perp) - 2 < \dim(\Omega_S) =\rk(S)-2$. Since there are countably many chains of primitive embeddings $N\hookrightarrow L\hookrightarrow \klat$, it follows that the points of $\Omega_S$ for which $N$ is a proper sublattice of $\NS_\ell$ form a union of countably many lower-dimensional subvarieties of $\Omega_S$. Thus, the very general point of $\Omega_S$ satisfies $N = \NS_\ell$, as required.

    Suppose now that we have another primitive sublattice $L\subset\klat$ of signature $(1,\rho)$ which is not equal to $N$ as a sublattice of $\klat$, and write $T = L^\perp$. By the above discussion, a very general point of $\Omega_T$ has $\NS_\ell = L$. In particular, we have $\Omega_T\neq \Omega_S$. To rephrase this, if $N$ and $L$ are two primitive sublattices of $\klat$, then we have $\Omega_{N^\perp} = \Omega_{L^\perp}$ if and only if $L = N$.
\end{remark}

\begin{proposition}\label{prop:densitycorankone}
    Let $L\subset \klat$ be a sublattice of signature $(1,\rho-1)$. Assume that $\rho\leq 9$. Suppose that $\left\{L_n\right\}_{n\in \mathbb{N}}$ is a set of pairwise non-isometric lattices of rank $(1,\rho)$, and for each $n\in \mathbb{N}$, we have a chain of primitive embeddings $L\hookrightarrow L_n\hookrightarrow \klat$.  Then the set $$\bigcup_{n\in \mathbb{N}}\calF_{L_n^\perp}$$ is dense in $\calF_{L^\perp}$.
\end{proposition}
\begin{proof}
    We first show that we have $\calF_{L_n^\perp}\neq \calF_{L_m^\perp}$ whenever $m\neq n$. We prove this by contradiction. Suppose that we have $\calF_{L_m^\perp} = \calF_{L_n^\perp}$ for some $m,n\in \mathbb{N}$ with $m\neq n$. Then we have 
    $$\bigcup_{g\in \toplus(L^\perp)}g(\Omega_{L_n^\perp})\cap \Omega_{L_m^\perp} = \Omega_{L_m^\perp}.$$
    Note that for any $g\in \toplus(L^\perp)$, 
    the subspace $g(\Omega_{L_n^\perp})\cap \Omega_{L_m^\perp}$ has codimension $1$ or $0$ in $\Omega_{L_m^\perp}$. Since $\toplus(L^\perp)$ is countable, 
    this means that there is a $g\in \toplus(L^\perp)$ such that 
    $g(\Omega_{L_n^\perp})\cap\Omega_{L_m^\perp}$ has codimension $0$, i.e. for which we have $g(\Omega_{L_n^\perp}) = \Omega_{L_m^\perp}.$ Since $g(\Omega_{L_n^\perp}) = \Omega_{g(L_n^\perp)},$ it follows from Remark \ref{rem:perioddomainknowsaboutthelattice} that we have $g(L_n^\perp) = L_m^\perp$. Since we have $g\in \toplus(L^\perp)$, we may extend $g$ to an isometry $f: \klat\simeq\klat$ with $f|_L = \operatorname{id}_L$. Moreover, since $g(L_n^\perp) = L_m^\perp$, we have $f(L_n) = L_m$, i.e. $f$ restricts to an isometry $L_n\simeq L_m$, which is a contradiction with the assumption that $L_n$ and $L_m$ are non-isometric. 
    
    To conclude, we use \cite[Theorem 3.8]{MP23}: The Euclidian closure of the union $$\bigcup_n\calF_{L_{n}^\perp}$$ is contained in a Shimura subvariety of $\calF_{L^\perp}$. Since all subspaces $\calF_{L_n^\perp}$ have codimension 1, this Shimura subvariety has to be all of $\calF_{L^\perp}$.
\end{proof}

\begin{lemma}\label{lem:divisibilityinfamilies}
    Let $(S,\phi)$ be a marked K3 surface.
    Let $v = (r,E,s)\in \operatorname{N}(S)$ be a primitive Mukai vector, and let $H\in \NS(S)$ be an ample divisor. Write $h = \phi(H)$ and $e = \phi(E)$. Let $L\subset \klat$ be the saturation of the sublattice $\langle h,e\rangle$. We consider the $L$-marked K3 surface $(S,\phi^{-1}|_L)$ as an element of the moduli space $\calF_{L^\perp}$. Then the set of points $[(X,i)]\in \calF_{L^\perp}$ for which 
    \begin{equation}\label{eq:divisibility}
    \dv_{\operatorname{N}(X)}(i(v))=1,
    \end{equation}
    where $i(v)\coloneqq (r,i(e),s)$,
    is dense in $\calF_{L^\perp}$.
\end{lemma}
\begin{proof}
    Let $e=me_0$ with $e_0$ primitive (in particular, $\gcd(r,m,s)=1$ since $v$ is primitive). We claim that there exists a vector $u\in \klat$ such that $u\cdot e_0 = 1$ and the lattice $\langle h,e,u\rangle$ has signature $(1,2)$. Indeed, let $N=U\oplus E_8$. Fix any primitive embedding $L\hookrightarrow N$, which is possible by \cite[Corollary 1.12.3]{Nik80}. Since $\rk(L)\leq 10$, the primitive embedding $L\hookrightarrow \klat$ is unique
    up to the action of $O(\klat)$. Therefore, we may assume that $L$ is a sublattice of $N\subset \klat$. Since $N$ is unimodular and $e_0$ is primitive, there exists a vector $u\in N$ such that $e_0\cdot u = 1$. 
    Now, for any $d\in E_8\subset N^\perp\subset \klat$, write $u_d=  u+d$, and let $L_d=\langle L,u_d\rangle\subset \klat$. If $d$ is primitive, then $L_d\subset \klat$ is a primitive sublattice. Moreover, by computing the discriminant of $L_d$, we see that $L_d$ is not isometric to $L_{d'}$ whenever $d^2\neq (d')^2$. Now, for any $n\in \mathbb{N}$, let $d_n\in E_8$ be a primitive vector with $d_n^2 = 2n$. Then $\left\{L_n\right\}_{n\in \mathbb{N}}$ is a set of pairwise non-isometric lattices satisfying the assumptions of Proposition \ref{prop:densitycorankone}, and therefore $\cup_n\calF_{L_n^\perp}$ is dense in $\calF_{L^\perp}$. Note that for any $n\in \mathbb{N}$ and any point $[(X,i)]\in \calF_{L_n^\perp}$, we have $\dv_{\N(X)}(i(v)) = 1$, as required.
\end{proof}

We now provide a proof for the following proposition, previously noted by Mukai \cite{Muk87} and C\u ald\u araru \cite{Cal00}.
\begin{proposition}\label{prop:deformuntwistedmoduli}
    Let $S$ be a K3 surface, let $v = (r,E,s)\in \N(S)$ be a primitive Mukai vector and let $H$ be a $v$-generic polarisation. Then there exists a proper, smooth morphism $\calS\to T$ of schemes of finite type over $\CC$ with connected fibres, such that for all $i\in \Z$, and all $t\in T$, we have an identification 
    \begin{equation} \label{eq:identificationcohomologyfamily}
    H^i(\calS_t,\Z)\simeq H^i(\calS,\Z).
    \end{equation}
    and such that:
    \begin{itemize}
        \item[i)] There exists a point $0\in T$ such that $\calS_0 \simeq S$.
        \item [ii)] The vector $v_t\in \tilde{H}(\calS_t,\Z)$ corresponding to $v$ along \eqref{eq:identificationcohomologyfamily} is contained in $\N(\calS_t)$ for all $t\in T$.
        \item [iii)] For all $t\in T$, $H_t\in H^2(\calS_t,\Z)$ is a $v$-generic polarisation.
        \item [iv)] There exists a point $1\in T$ such that $\dv_{\N(\calS_1)}(v)=1.$ 
    \end{itemize}
    Moreover, there is a smooth, proper family $\calM\to T$ with the following properties:
    \begin{itemize}
        \item [v)] For all $t\in T$, there exists an isomorphism $\calM_t\simeq M_{H_t}(v_t)$.
        \item [vi)] There is a Brauer class $\alpha\in \Br(\calM)$ and a $(1\boxtimes\alpha)$-twisted sheaf $\calE$ on $\calS\times_T\calM$ which restricts to a twisted universal sheaf on each fibre $\calS_t\times\calM_t$.
    \end{itemize}
\end{proposition}
\begin{proof}
   Choose any marking $\phi: H^2(S,\Z)\simeq \klat$ and let $h = \phi(H)$, $e = \phi(E)$, and let $L$ be the saturation of $\langle h,e\rangle$ in $\klat$. Then $(S, \phi^{-1}|_{L})$ is an $L$-marked K3 surface, hence it represents a point $0\in \calF_{L^\perp}$. Let $0\in T\subset \calF_{L^\perp}$ be a small open neighbourhood, and let $\calS\to T$ be the tautological family of K3 surfaces. By shrinking $T$, for example until it is contractible, we may assume that we have the identifications \eqref{eq:identificationcohomologyfamily}. Moreover, by shrinking $T$ further, we may assume that $H_t$ is ample, and that $H_t$ is $v_t$-generic for all $t\in T$ \cite{Bri08}. Note that $H_t$ and $v_t$ are automatically algebraic, since we have $T\subset \calF_{L^\perp}$. The existence of the relative moduli space $\calM\to T$ and the relative twisted universal sheaf follows from Proposition \ref{prop:relativetwistedmodulispaces} and Definition/Proposition \ref{defprop:universaltwistedsheaves}. 
   Part \textit{iv)} follows from Lemma \ref{lem:divisibilityinfamilies}.
\end{proof}

\section{Identifying the obstruction}
In this section, we compute the obstruction Brauer class for a higher-dimensional moduli space of sheaves on a K3 surface in terms of the Mukai vector. One of the main points of interest is that we compute the order of the obstruction class precisely. 

\subsection{Brauer groups of moduli spaces}
In this section, let $S$ be a K3 surface with primitive Mukai vector $v\in \N(S)$. Assume $v^2 >0$. Let $H\in \NS(S)$ be a $v$-generic polarisation, and write $M\coloneqq M_H(v)$. The main goal of this subsection is to show that there is an exact sequence resembling \eqref{eq:sesbrauerfork3} in this setting. We note that, since the Mukai morphism induces a Hodge isometry $v^\perp\simeq H^2(M,\Z),$ we have $T(S)\simeq T(M)$, hence there is no short exact sequence of the form \eqref{eq:sestranscendentalfork3s}. Instead, the correct first step is to recreate short exact sequence \eqref{eq:sestranscendentalduals} using the group $$T'(M)\coloneqq \frac{H^2(M,\Z)}{\NS(M)}$$ from \eqref{eq:Tprime}.

\begin{lemma}\label{lem:sestranscendental}
    Let $S$, $v$, $H$, and $M$ be as above. Then there is a short exact sequence 
    \begin{equation}
        0\to T'(M)\to T(S)^*\to \Z/\operatorname{div}(v)\Z\to 0.
    \end{equation}
    Let $u\in \tilde{H}(S,\Z)$ such that $u\cdot v \equiv 1\pmod{\dv(v)}$, then the image of $u$ under the projection $H^2(S,\Z)\to T(S)^*$ is a generator for the cokernel.
\end{lemma}
\begin{proof}
    The Mukai morphism is a Hodge isometry $H^2(M,\Z)\simeq v^\perp \subset \tilde{H}(S,\Z)$ by Theorem \ref{thm:mukaimorphism}. This induces an embedding of Hodge structures 
    \begin{equation}\label{eq:transcendentalembedding}
        T'(M) \simeq \frac{v^\perp}{(v^\perp)^{1,1}}\hookrightarrow \frac{\tilde{H}(S,\Z)}{\operatorname{N}(S)}\simeq T(S)^*.
    \end{equation}
    It follows from a straightforward computation that the cokernel of this embedding is generated by $u$: for any class $x\in\tilde{H}(S,\Z)$ such that $x\cdot v=\lambda$, we have $(x-\lambda u)\cdot v\equiv0\pmod{\dv(v)}$. Therefore, there exists an algebraic class $y\in \N(S)$ such that $y\cdot v = (x-\lambda u)\cdot v$. Therefore, we have $x-\lambda u - y \in v^\perp$, hence $x$ and $\lambda u$ induce the same element in $T(S)^*/T'(M)$. Since $\operatorname{div}(v)u-z\in v^\perp$ for any $z\in \N(S)$ with $v\cdot z=\operatorname{div}(v)$, we obtain that $u$ has order $\operatorname{div}(v)$ in $T(S)^*/T'(M)$.
\end{proof}
Let $w\in T'(M)$ be the vector which maps to $\dv(v)\cdot u\in T(S)^*$. Using the description of $\Br(M)$ as $\Br(M)\simeq T'(M)\otimes \Q/\Z$ from \eqref{eq:brauergroup}, Lemma \ref{lem:sestranscendental} immediately implies the following.

\begin{proposition} \label{prop:sesbrauerkn}
    There is a short exact sequence $$0\to \left\langle \frac{w}{\dv(v)}\right\rangle \to \Br(M)\to \Br(S)\to 0.$$
\end{proposition}

We will see in the next section that $\frac{w}{\dv(v)}$ is precisely the Brauer class of $M$ that obstructs the existence of a universal sheaf on $S\times M$.

The following lemma gives us another way to view the obstruction class in terms of the Mukai morphism.

\begin{lemma}\label{lem:mukaimorphismcharacterisationoftheobstruction}
    Let $\varphi: \tilde{H}(S,\Q) \to H^*(M,\Q)$ be the Mukai morphism from Definition \ref{def:mukaimorphism}. Consider the vector $[\varphi(u)]_2\in H^2(M,\Q)$, where $u\in \tilde{H}(S,\Z)$ is as in Lemma \ref{lem:sestranscendental}. Then the Brauer class on $M$ induced by $[\varphi(u)]_2$ along the composition $$H^2(M,\Q)\to T'(M)\otimes \Q\to T'(M)\otimes \Q/\Z\simeq \Br(M)$$ is precisely $\frac{w}{\dv(v)}$. Moreover, this is independent of the choice of a quasi-universal sheaf on $S\times M$ and the choice of $u$.
\end{lemma}
\begin{proof}
    The fact that $[\varphi(u)]_2$ induces the Brauer class $\frac{w}{\dv(v)}$ is a straightforward chase through the identifications.
    We check that the Brauer class $[\varphi(u)]_2$ is independent of our choice of quasi-universal sheaf $\calU$. Recall that if $\calU'$ is any other quasi-universal sheaf on $S\times M$, then there exist vector bundles $E$ and $F$ on $M$ such that $\calU\otimes p^*E\simeq \calU'\otimes p^*F$ by \cite[Appendix 2]{Muk87}.
    For $\calF\in\Coh(M)$ we have $$[\varphi^{\calU^\vee\otimes p^*\calF}(u)]_2 = [\varphi^{\calU^\vee}(u)]_2+[\varphi^{\calU^\vee}(u)]_0\cdot \frac{c_1(\calF)}{\rk\calF}.$$ Since $$\frac{c_1(\calF)}{\rk\calF}\in \NS(M)\otimes \Q,$$ this shows that the class of $[\varphi(u)]_2$ in $\Br(M)\simeq T'(M)\otimes \matQ/\matZ$ does not depend on the choice of $\calU$.

    Now we check that $[\varphi(u)]_2$ is independent of the choice of $u$. Suppose that we have $u'\in \tilde{H}(S,\Z)$ such that $u\cdot v = u'\cdot v \equiv 1\pmod{\dv(v)}$. Then $(u-u')\cdot v\equiv 0\pmod{\dv(v)}$. Therefore, there is an algebraic class $x\in \tilde{H}(S,\Z)$ such that $x\cdot v = (u-u')\cdot v$. This means that $u-u'-x\in v^\perp$. This implies that $[\varphi(u-u'-x)]_2\in H^2(M,\matZ)$ is integral, so in particular it vanishes in $\Br(M)$. Since $\varphi$ is a morphism of Hodge structures, $\varphi(x)$ is algebraic, and we obtain $[\varphi(u)]_2 - [\varphi(u')]_2 = -[\varphi(x)]_2=0\in \Br(M)$. 
\end{proof}

\subsection{Obstruction classes}
\label{sec:identifyingobstructionclass}
The main goal of this section is to prove Theorem \ref{thm:explicitobstruction}, which computes the obstruction Brauer class explicitly when $M$ is a higher-dimensional moduli space of sheaves. Our strategy for proving Theorem \ref{thm:explicitobstruction} is very similar to the proof of \cite[Theorem 1.1]{Cal02}, but some modifications need to be made to the proof in our setting. 
The main idea behind the proof is the following result by C\u ald\u araru.

\begin{theorem}\normalfont{\cite[Theorem 4.1]{Cal02}}\label{thm:CaldaIdentificationBrauerClassFamily}
    Let $f: X\to T$ be a proper, smooth morphism of analytic spaces. Let $\mathcal{E}$ be a locally free $\alpha$-twisted sheaf on $X$. Assume that $\alpha|_{X_0}$ is trivial, so that $\mathcal{E}_0\coloneqq \mathcal{E}|_{X_0}$ can be glued to an untwisted sheaf on $X_0$. Assume that for all $t\in T$, we have $H^i(X,\Z)\simeq H^i(X_t,\Z)$. Then  $$\alpha = -\frac{c_1(\mathcal{E}_0)}{\operatorname{rk}(\mathcal{E}_0)} \in \frac{H^2(X,\Z)}{\NS(X)}\otimes \Q/\Z.$$
\end{theorem}

\begin{theorem}\label{thm:explicitobstruction}
    Keeping the notation as in Lemma \ref{lem:sestranscendental} and Proposition \ref{prop:sesbrauerkn}, the Brauer class $$\frac{w}{\dv(v)}\in \Br(M)$$ is the obstruction to the existence of a universal sheaf on $S\times M$.
\end{theorem}
\begin{proof}
We write $v = (r,E,s)\in \operatorname{N}(S)$.
Let $\calS\to T$ and $\calM\to T$ be the families of K3 surfaces and moduli spaces of Proposition \ref{prop:deformuntwistedmoduli}, and let $\alpha\in \Br(\calM)$ be the Brauer class, and $\calE$ the $(1\boxtimes \alpha)$-twisted universal sheaf on $\calS\times_T \calM$ from the same token. We identify $\calS_0 \simeq S$ and $\calM_0\simeq M$.

Pick a locally free $\alpha^{-1}$-twisted sheaf $\calF$ on $\calM$ of rank $\rho$. Fix $\calU\coloneqq \calE\otimes p^*\calF$, which is a relative quasi-universal sheaf on $\calS\times_T \calM$ of similitude $\rho$. For $t\in T$, we consider the induced Fourier-Mukai transforms 
$$\Phi_t\coloneqq \Phi^{\calU_t^\vee}\colon \Db(S_t) \to \Db(M_t)$$
and morphisms of Hodge structures:

$$
 \begin{array}{rl@{\qquad}lr}
	\varphi_t\coloneqq \varphi^{(1/\rho) v(\calU_t^\vee)}: {H}^*(S_t,\Q)&\to H^*(M_t,\Q)  \\
	x&\mapsto \frac{1}{\rho}\cdot p_*\left(v(\calU_t^\vee)\cdot q^*(x)\right).
\end{array}
$$
satisfying 
$v\left( \Phi_t(-)\right)=\rho\varphi_t\left(v(-)\right)$.
By shrinking $T$, we may assume that $\varphi_t$ is a constant map on $t$.

Let $u\in \tilde{H}(S,\Z)$ be a vector such that $u\cdot v = 1 \pmod {\dv(v)}$. Note that the class $[\varphi_0(u)]_2 \in H^2(M,\Q)$ induces a Brauer class on $M$ via the natural surjection $$H^2(M,\Q)\to T'(M)\otimes \Q \to T'(M)\otimes \Q/\Z\simeq \Br(M).$$ By a slight abuse of notation, we also denote this Brauer class by $[\varphi_0(u)]_2.$ By Lemma \ref{lem:mukaimorphismcharacterisationoftheobstruction}, we wish to prove that $$\alpha|_{M} = [\varphi_0(u)]_2 \in \Br(M).$$

On the fibre over $1\in T$, the Brauer class $\alpha_1$ vanishes by Proposition \ref{prop:fineifdivone}, so that $\calE_1$ is a universal sheaf. In particular, $\Phi'\coloneqq \Phi^{\calE_1^\vee}$ and $\varphi'\coloneqq \varphi^{v(\calE_1^\vee)}$ are well-defined. For any $\calG\in \Db(S)$ and $x\in \tilde{H}(S,\Q)$, we have
$\Phi_1(\calG)=\Phi'(\calG)\otimes \calF^\vee_1$ and 
\begin{equation}\label{eq:phiprimeandphi}
\varphi_1(x)=\varphi'(x)\dfrac{\ch(\calF^\vee_1)}{\rho},
\end{equation}
due to the projection formula. In particular, we have $$[\varphi_1(u)]_2=[\varphi'(u)]_2 + [\varphi'(u)]_0\cdot \dfrac{-c_1(\calF_1)}{\rho}.$$
Since this equation is purely topological, it also holds on the fibre at $0$. We know by Theorem \ref{thm:CaldaIdentificationBrauerClassFamily} that $$\alpha|_M = \frac{c_1(\calF_1)}{\rho}$$ 
(recall that $\calF$ is $\alpha^{-1}$-twisted). 
%Therefore, we must show that $[\varphi'(u)]_0=1$ and $[\varphi'(u)]_2\in H^2(M_1,\matZ)$.

Firstly, note that we have $[\varphi'(u)]_0=-1$ as a consequence of Lemma \ref{lem:mukaipairingmoduli}, since we are assuming that $u\cdot v=1$. 

Secondly, we show that $[\varphi'(u)]_2$ is integral. For this, we first show that we can find a line bundle $\calL$ on $S_1$ and choose $u=v(\calL)$. 

Indeed, since the divisibility of $v$ in $\operatorname{N}(S_1)$ is 1, we can find a divisor $D\in \NS(S_1)$ such that $\operatorname{gcd}(r,D\cdot E,s)=1$. Let $\calL = \bigo_{S_1}(D)$. Then $v(\calL) = (1,D,\frac{D^2}{2}-1)$, and we have $v(\calL)\cdot v = -r\cdot (\frac{1}{2}D^2 - 1) - s + D\cdot E.$ Note that $r$ and $s$ are divisible by $\dv_{\N(S)}(v)$, hence by replacing $D$ by $mD$ for some $m\in \Z$ we obtain $v(\calL)\cdot v\equiv 1\pmod{\dv_{\N(S)}(v)}$.

We now show that, by replacing $\calL$ with $\calL(k)\coloneqq \calL\otimes \bigo_{S_1}(k)$ for $k\gg 0$, we may assume that $\calG\coloneqq \Phi'(\calL)$ is locally free. It suffices to show that $R^jp_*(q^*\calL(k)\otimes \calE^\vee_1)=0$ for all $j>0$ and $k\gg0$. For any point $[\calK] \in M_1$, we have natural maps $$R^jp_*(q^*\calL(k)\otimes \calE^\vee_1)_{[\calK]} \to H^j(S_1,\calK\otimes \calL(k)).$$ 
By boundedness of $M_1$, all these cohomology groups vanish  for all $[\calK]\in M_1$, $j>0$ and $k\gg 0$, which in turn implies that the $R^jp_*(q^*\calL(k)\otimes \calE^\vee_1)_{[\calK]}$ vanish as well.

Note that this completes the proof, since we have $$[\varphi'(u)]_2=[v(\calG)]_2=[c_1(\calG)]\in H^2(M_1,\matZ).$$
\end{proof}

\begin{corollary}\label{cor:fineiffdivone}
   Let $S$ be a K3 surface and let $v\in \operatorname{N}(S)$ be a Mukai vector with $v^2 >0$. Let $H$ be a $v$-generic polarisation of $S$. Then the moduli space $M_H(v)$ is fine if and only if $\operatorname{div}(v)=1$.
\end{corollary}
For isotropic Mukai vectors, the conclusion of Corollary \ref{cor:fineiffdivone} follows immediately from the work of C\u ald\u araru \cite{Cal00}, but it also follows directly from the Derived Torelli Theorem \ref{thm:introderivedtorelli}. In the higher-dimensional case, the fact that $M_H(v)$ is fine whenever $\operatorname{div}(v)=1$ was already known, c.f. Proposition \ref{prop:fineifdivone}. The new part of Corollary \ref{cor:fineiffdivone} is that the divisibility of the Mukai vector is $1$ for higher-dimensional fine moduli spaces.

\section{Beauville--Mukai systems}

Let $S$ be a K3 surface, and let $H\subset S$ be a smooth, irreducible curve of genus $g>0$, whose class in $\NS(S)$ is primitive. 
Let $\calC\to |H|$ be the universal curve over the linear system $|H|$. 
This means that, for any $[C]=x\in |H|$, there is an isomorphism $\calC_x\simeq C$.
For $d\in \Z$, consider the Mukai vector 
\begin{equation}\label{eq:vd}
    v_d \coloneqq (0,H,d+1-g).
\end{equation}
For a $v_d$-generic polarisation $H'\in \NS(S)$, we write $M_d = M_{H'}(v_d)$. 
Note that a line bundle $L$ of degree $d$ on a smooth, irreducible curve $C\in |H|$ satisfies $$v(i_*L) = (0,H,d+1-g),$$ where $i:C\hookrightarrow S$ is the inclusion.
Therefore, we have an inclusion $$\Pic^d(\calC_{\operatorname{sm}}/|H|_{\operatorname{sm}})\hookrightarrow M_d,$$ where $\calC_{\operatorname{sm}}\to |H|_{\operatorname{sm}}$ is the restriction of the universal curve to the locus of smooth curves in $|H|$. 
For this reason, we usually denote $$\olPic^d(\calC/|H|)\coloneqq M_d.$$
To keep the notation light, we often omit $\calC/|H|$ from the notation if this cannot lead to confusion. Note that, since $H$ is primitive in $\NS(S)$, the Mukai vector $v_d\in \N(S)$ is primitive for all $d\in \Z$. Hence, by Theorem \ref{thm: moduli sheaves K3}, $\olPic^d$ is a smooth hyperk\"ahler manifold of $\text{K3}^{[g]}$-type.

Note that there is a natural morphism 
$$\fonction{f}{\olPic^d}{|H|}{[\calF]}{\operatorname{Supp}(\calF),}$$
which is a Lagrangian fibration. Moreover, for a smooth curve $C\in |H|$, the fibre of $f$ over $C$ is isomorphic to $\Pic^d(C)$.

Letting $\eta\in |H|$ be the generic point, the generic fibre $\olPic^0_\eta$ is an abelian variety over $\spec(\CC(\eta))$, and for any $d\in \Z$, the generic fibre $\olPic^d_\eta$ is a torsor under $\olPic^0_\eta$.  

Define $\zeta_H: A_{\NS(S)}\to \Z/\dv(H)\Z$ by $$\zeta_{H}(\ol{a}) = a\cdot H\pmod {\dv(H)},$$ where $\ol{a}$ denotes the image of $a\in \NS(S)^*$ in $A_{\NS(S)}$. We note that this map is well-defined, since for any $a'\in \NS(S)^*$ such that $\ol{a'}=\ol{a}$, there exists a class $x\in \NS(S)$ such that $a'+x = a$, hence $a\cdot H = (a'+x)\cdot H = a'\cdot H \pmod{\dv(H)}$.
Recall from Lemma \ref{lem:intersectingspecialbrauer} that we may view $A_{T(S)}\simeq A_{\NS(S)}(-1)$ as a subgroup of $T(S)\otimes \Q/\Z$ via the inclusion $T(S)^*\hookrightarrow T(S)\otimes \Q$. 

\begin{definition}\cite{HM23}
    The \textit{Tate--Shafarevich group} of the pair $(S,H)$ is the group $$\Sha(S,H)\coloneqq \frac{T(S)\otimes \Q/\Z}{\ker(\zeta_H)}.$$
\end{definition}

The reason $\Sha(S,H)$ is called the Tate--Shafarevich group is that it parametrises certain torsors of the abelian variety $\olPic^0_\eta$. It generalises the Tate--Shafarevich group of an elliptic K3 surface. More precisely, if $S\to \PP^1$ is an elliptic K3 surface that admits a section, and $F\in \NS(S)$ denotes the fibre class of the elliptic fibration, then there is an isomorphism $\Sha(S,F)\simeq \Sha(S)$ \cite[\S 5.1]{HM23}. 

For any $d\in \Z$, the Beauville--Mukai system $\olPic^d\to |H|$ defines an element of $\Sha(S,H)$. For details, we refer to \cite[Example 4.14]{HM23}. The element of $\Sha(S,H)$ corresponding to $\olPic^d$ is the element induced by 
\begin{equation}\label{eq:picdinshafa}
    \frac{-d}{\dv(D)}\cdot D \in A_{\NS(S)}\subset T(S)\otimes \Q/\Z,
\end{equation}
where $D\in \NS(S)$ is a divisor with $\dv(D) = D\cdot H$. Such a divisor exists and can be constructed as follows. Let $u\in H^2(S,\Z)$ be a class that satisfies $u\cdot H = 1$. This class exists since $H$ is primitive and $H^2(S,\Z)$ is unimodular. Also by unimodularity, there exists a divisor $D\in \NS(S)$, and an integer $m\in \Z$ such that $u\cdot E = \frac{D}{m}\cdot E$ for all $E\in\NS(S)$. Moreover, we have $\frac{D}{m}\in \NS(S)^*$, since $u$ is an integral vector, hence $m\leq\dv(D)$. On the other hand, from $\frac{D}{m}\cdot H = 1$, it follows that we have $m = D\cdot H$, hence $m\geq \dv(D)$, and we find $m = \dv(D)$.

\begin{remark}
    Let us briefly explain our choice of $D$ above. For simplicity, assume $d=1$. In the notation of \cite{HM23}, for any vector bundle $\calE$ on $S$ of rank $r$ and determinant $L$, there exists an isomorphism 
    $$\isomorphismstar{\Pic^1}{\Pic^{d'}_{\SEnd(\calE)}}{F}{F\otimes\calE}$$
    where $d'=1+L\cdot H/r$. Moreover, the class $[\SEnd(\calE)]\in \SBr(S)$ corresponds to the class $(1/r)\cdot L\in \NS(S)\otimes \bbQ/\bbZ$. Therefore, for any divisor $D\in \NS(S)$, one can pick $L=-D$ and $r=D\cdot H$ to get an isomorphism $\Pic^d\simeq \Pic^0_{\SEnd(\calE)}$. However, $\SEnd(\calE)$ only defines a class in $\Sha(S,H)$ if $[\SEnd(\calE)]$ lies in the image of $A_{\NS(S)}\subset \NS(S)\otimes \bbQ/\bbZ$, i.e. if $(-1/D\cdot H)\in \bbZ[1/\dv(D)]$, equivalently $D\cdot H=\dv(D)$. 

\end{remark}

For now, let us note that $[\olPic^1]$ corresponds to the image of $u\in H^2(S,\Z)$ along the composition $$H^2(S,\Z)\to T(S)^*\to T(S)\otimes \Q/\Z.$$

\begin{remark}\label{rem:whichdegreesgiveisomorphictorsors}
    The element \eqref{eq:picdinshafa}, corresponding to $[\olPic^d]\in \Sha(S,H)$, is trivial if and only if $d$ is a multiple of $\dv(H)$. Indeed, since $\dv(D)=D\cdot H=k\dv(H)$ for some $k$, we have 
    $$\zeta_H\left(\dfrac{-dD}{\dv(D)}\right)=\dfrac{-d\dv(D)}{\dv(D)}=\overline{-d}\in \bbZ/\dv(H)\bbZ.$$
    This means that $\olPic^d_\eta$ and $\olPic^e_\eta$ are isomorphic $\olPic^0_\eta$-torsors if and only if $d -e\equiv0\pmod{\dv(H)}$. Notice that, in particular, $[\olPic^1]=[-D/\dv(D)]$ has order $\dv(H)$.
\end{remark}

\subsection{A Donagi--Pantev type result}
Recall the following result by Addington, Donovan, and Meachan:
\begin{theorem}\cite{ADM16}\label{thm:ADM}
    Let $S$ be a K3 surface with $\NS(S) = \langle H\rangle$, where $H$ is ample. Then for any $d,e\in \Z$, there is an equivalence $$\Db(\olPic^d,\alpha_d^e)\simeq \Db(\olPic^e,\alpha_e^{-d}).$$
\end{theorem}

The main goal of this subsection is to exhibit this result as a higher-dimensional generalisation of a result by Donagi and Pantev of \cite{DP08}. This is done by Corollary \ref{cor:adm=dp} below. The statement and proof of Corollary \ref{cor:adm=dp} heavily rely on our explicit computation of the obstruction class in Theorem \ref{thm:explicitobstruction}.

\subsubsection{Twisted derived equivalences of elliptic K3 surfaces}
First, we briefly recall the results of Donagi and Pantev. Their results hold more generally for elliptic surfaces $S\to \PP^1$ whose fibres have at worst $I_1$-singularities, but we restrict our attention to elliptic K3 surfaces. Let $S\to \PP^1$ be an elliptic K3 surface which admits a section. Recall that the Tate--Shafarevich group $\Sha(S)$ of $S$ parametrises pairs $(X\to \PP^1,\phi)$, where $X\to \PP^1$ is an elliptic K3 surface, and $\phi: \Jac^0(X)\simeq S$ is an isomorphism over $\PP^1$ which preserves the chosen sections, see \cite[\S11.5]{Huy16} for details. Elements of $\Sha(S)$ are called $S$-torsors. Moreover, there is a canonical isomorphism 
\begin{equation}\label{eq:isom--sha--br}
    \Sha(S)\simeq \Br(S).
\end{equation}
This isomorphism can be understood as follows. Suppose $(X\to \PP^1,\phi)$ is an $S$-torsor. Then $\phi$ exhibits $S$ as a coarse moduli space of sheaves on $X$, hence there is a unique Brauer class $\alpha_X\in \Br(S)$ which is the obstruction to the existence of a universal sheaf on $X\times S$ by Definition/Proposition \ref{defprop:universaltwistedsheaves}. The isomorphism  \eqref{eq:isom--sha--br} is given by 
\begin{equation*}
    \functionstar{\Sha(S)}{\Br(S)}{(X\to \PP^1,\phi)}{\alpha_X}.
\end{equation*}
Moreover, short exact sequence \eqref{eq:sesbrauerfork3} takes the following form in this setting:
\begin{equation*}
    0\to \langle \alpha_X\rangle \to \Br(S)\to \Br(X)\to 0,
\end{equation*}
and via \eqref{eq:isom--sha--br}, this sequence may be interpreted as the sequence 
\begin{equation}\label{eq:ses--shas--brx}
    0 \to \langle[(X\to\PP^1,\phi)]\rangle \to \Sha(S)\to \Br(X)\to 0.
\end{equation}

Let $(Y\to \PP^1,\psi)\in \Sha(S)$ be another $S$-torsor, and denote by $\ol{\alpha_Y}\in \Br(X)$ the image of $[(Y\to\PP^1,\psi)]\in \Br(S)$ in $\Br(X).$ Similarly, we denote by $\ol{\alpha_X}\in \Br(Y)$ the image of $\alpha_X$. 

\begin{theorem}\cite{DP08}\label{thm:donagi--pantev}
    Keeping the notation as above, there is an equivalence $$\Db(X,\ol{\alpha_Y})\simeq \Db(Y,\ol{\alpha_X}^{-1}).$$
\end{theorem}

\subsubsection{Tate--Shafarevich groups of Beauville--Mukai systems}
We now turn our attention back to Beauville--Mukai systems. Let $S$ be a K3 surface.
For $M$ a smooth moduli space of sheaves on $S$, recall that there is a natural chain of inclusions $T(S)\hookrightarrow T'(M)\hookrightarrow T(S)^*.$ We denote $A\coloneqq T'(M)/T(S)$, and consider it as a subgroup of $A_{T(S)} \simeq T(S)^*/T(S).$ Using this, we obtain the short exact sequence
\begin{equation}\label{eq:sespreshafa}
    0\to A\to T(S)\otimes \Q/\Z\to T'(M)\otimes \Q/\Z\to 0,
\end{equation}
where $A$ is embedded in $T(S)\otimes \Q/\Z$ using the inclusion $A\subset A_{T(S)}\hookrightarrow T(S)\otimes \Q/\Z$.  
\begin{proposition}\label{prop:sesshafabrauer}
    Let $S$ be a K3 surface, and let $H\subset S$ be a smooth, irreducible curve whose class in $\NS(S)$ is primitive, and such that $H^2 = 2g-2\geq 2$. If we take $M = \olPic^d$ in short exact sequence \eqref{eq:sespreshafa}, we obtain the short exact sequence:
   \begin{equation}\label{eq:ses-sha--br}
        0\to \frac{A}{\ker(\zeta_H)}\to \Sha(S,H) \to \Br(\olPic^d)\to 0.
    \end{equation} 
    Moreover, the kernel is cyclic of order $$\frac{\dv(H)}{\dv(v_d)} = \frac{\dv(H)}{\gcd(\dv(H),d+1-g)}.$$ Here, $v_d$ is the Mukai vector of $\olPic^d$, see \eqref{eq:vd}.
\end{proposition}
\begin{proof}
    Recall from the previous section that we have $$\im(T'(\olPic^d)\hookrightarrow T(S)^*) = \left\{a\in T(S)^*\mid a\cdot v_d\in \dv(v_d)\Z\right\}.$$ 
    This implies that $\ker(\zeta_H)\subset A$, hence short exact sequence \eqref{eq:ses-sha--br} is obtained from \eqref{eq:sespreshafa} by taking the quotient with $\ker(\zeta_H)$ and using the third isomorphism theorem. 
    Choose any $a\in A$ such that $a\cdot v_d = \dv(v_d)$. We claim that $a$ generates $A/\ker(\zeta_H)$. Let $b\in A$ be any other element, then we can write $b\cdot H = \dv(v_d)\cdot k$ for some $k\in \Z$. Since $ka-b\in A$ satisfies $(ka-b)(H) = 0$, it follows that $ka = b$ in $A/\ker(\zeta_H)$. This shows that $a$ generates $A/\ker(\zeta_H)$. Moreover, the order of $a$ is precisely $\dv(H)/\dv(v_d)$. 
\end{proof}

\begin{corollary}\label{cor:isom-sha--brg-1}
    In the notation of Proposition \ref{prop:sesshafabrauer}, we have an isomorphism $$\Sha(S,H)\simeq \Br(\olPic^{g-1}).$$
\end{corollary}
\begin{proof}
    Since $v_{g-1} = (0,H,0)$ it follows that $\dv(v_{g-1}) = \dv(H)$, and now the statement follows from Proposition \ref{prop:sesshafabrauer}.  
\end{proof}

Corollary \ref{cor:isom-sha--brg-1} is slightly surprising. One might have expected an isomorphism between $\Sha(S,H)$ and $\Br(\olPic^0)$, as is the case when $H^2=0$. Instead, if the Picard rank of $S$ is 1, we have a short exact sequence $$0\to \Z/2\Z\to \Sha(S,H)\to \Br(\olPic^0)\to 0$$ by Proposition \ref{prop:sesshafabrauer}. Note, however, that if $H^2 = 0$, we have $g=1$, so that $\olPic^0 = \olPic^{g-1}$. Therefore, Corollary \ref{cor:isom-sha--brg-1} is a generalisation of the well-known isomorphism $\Sha(S)\simeq \Br(S)$ for an elliptic K3 surface $S$ with a section.

Since we have, for any $d\in \Z$, a surjective group homomorphism $\Sha(S,H)\to \Br(\olPic^d)$, it is an important question what the preimages of the obstruction class are. We answer this question in the following proposition, using Theorem \ref{thm:explicitobstruction}.

\begin{proposition}\label{prop:piconemapstoobstruction}
    For any $d\in \Z$, the group homomorphism $\Sha(S,H)\to \Br(\olPic^d)$ of Proposition \ref{prop:sesshafabrauer} satisfies $$[\olPic^1]\mapsto \alpha_d,$$ where $\alpha_d\in \Br(\olPic^d)$ is the obstruction to the existence of a universal sheaf on $S\times \olPic^d$.
\end{proposition}
\begin{proof}
Recall that the element $[\olPic^1]\in \Sha(S,H)$ is the image of the element $$\ol{u}\in T(S)\otimes \Q/\Z,$$ where $u\in H^2(S,\Z)$ is a class with $u\cdot H = 1$. We show that the image of $\overline{u}$ in $\Br(\olPic^d)$ is precisely the obstruction class. Consider the chain of surjections 
\begin{equation*}
    \begin{array}{rrr}
         T(S)\otimes \Q/\Z &\twoheadrightarrow T'(\olPic^d)\otimes \Q/\Z&\twoheadrightarrow T(S)^*\otimes \Q/\Z.  \\
         & \simeq \Br(\olPic^d) & \simeq \Br(S)
    \end{array}
\end{equation*}
Since $\ol{u}\in T(S)^*$, it follows that the image of $\ol{u}$ in $\Br(\olPic^d)$ is in the kernel of the map $\Br(\olPic^d)\to \Br(S)$, which is generated by $\alpha_d$. From $u\cdot v_d = 1$, it follows that $\overline{u}$ maps to $\alpha_d$ by Theorem \ref{thm:explicitobstruction}.
\end{proof}

We can now see how Theorem \ref{thm:ADM} resembles Theorem \ref{thm:donagi--pantev}. 

\begin{corollary}\label{cor:adm=dp}
    Let $S$ be a K3 surface of Picard rank 1. Let $H\in \NS(S)$ be an ample generator, and let $d,e\in \Z$. Denote by $\overline{\alpha_d}\in \Br(\olPic^e)$, resp. $\ol{\alpha_e}\in \Br(\olPic^d)$ the images of $[\olPic^d]$, resp. $[\olPic^e]$ in $\Br(\olPic^e)$, resp. $\Br(\olPic^d)$. Then there is an equivalence $$\Db(\olPic^d,\ol{\alpha_e})\simeq \Db(\olPic^e,\ol{\alpha_d}^{-1}).$$
\end{corollary}
\begin{proof}
By Proposition \ref{prop:piconemapstoobstruction}, we have $\ol{\alpha_d} = \alpha_e^d\in \Br(\olPic^e)$, and $\ol{\alpha_e} = \alpha_d^e\in \Br(\olPic^d)$. Combining this with Theorem \ref{thm:ADM} proves the result.
\end{proof}

It is an interesting open question whether Corollary \ref{cor:adm=dp} holds more generally. For example, when we consider the twisted Picard varieties of \cite{HM23}, or when we drop the assumption that $\rho = 1$. 

\subsection{Birational equivalence of moduli spaces}
In this subsection, we study birational equivalences of moduli spaces of sheaves on elliptic K3 surfaces with a section. More specifically, we study the birational geometry of Beauville--Mukai systems on such K3 surfaces. Theorem \ref{thm:birationalequivalencebeauville--mukaioverelliptick3surfaces} below fully classifies which Beauville--Mukai systems over elliptic K3 surfaces are birational. This classification uses our explicit computation of the obstruction class. 

It is a well-known fact that a K3 surface $S$ admits an elliptic fibration with a section if and only if there is an isotropic class $F\in \NS(S)$ of divisibility $1$, see for example \cite[Remark 3.2.13 and \S11.4]{Huy16}. This is equivalent to the existence of an embedding $U\subset \NS(S)$. In this case, the lattice $\N(S) \simeq \NS(S)\oplus U$ admits an embedding $U^{\oplus2}\subset \N(S)$, so that we may apply Lemma \ref{lem:eichlertransvections} to $\N(S)$.

The main result about the birational geometry of $M_H(v)$ is the Markman's \textit{Birational Torelli Theorem}. We write $\tilde{\Lambda}\coloneqq U\oplus \klat = U^{\oplus 4}\oplus E_8(-1)^{\oplus 2}$. Note that, as an abstract lattice, $\tilde{\Lambda}$ is isometric to the Mukai lattice of any K3 surface. For any hyperk\"ahler variety of \Kn-type, there is a natural $O(\tilde{\Lambda})$-orbit $i_X$ of embeddings $H^2(X,\Z)\hookrightarrow \tilde{\Lambda}$ \cite{Mar11}. 

\begin{theorem}[Birational Torelli Theorem] \cite[Corollary 9.9]{Mar11}\label{thm:birationaltorelli}
    Let $X$ and $Y$ be hyperk\"ahler varieties of \Kn-type. The following are equivalent:
    \begin{itemize}
        \item[i)] $X$ and $Y$ are birational,
        \item[ii)] There is a Hodge isometry $H^2(X,\Z)\simeq H^2(Y,\Z)$ making the following diagram commute: \begin{equation*}
            \xymatrix{
            H^2(X,\Z) \ar[r] \ar[d] & H^2(Y,\Z) \ar[d]\\
           \tilde{\Lambda}\ar[r] & \tilde{\Lambda},
            }
        \end{equation*}
        where the vertical arrows are embeddings contained in the orbits $i_X$ and $i_Y$, respectively.
    \end{itemize}
\end{theorem}

If $M = M_H(v)$ is a moduli space of sheaves on a K3 surface $S$, then the orbit of embeddings $H^2(M,\Z)\hookrightarrow \tilde{\Lambda}$ is the one containing the following embedding: 
\begin{equation}\label{eq:naturalembeddingformoduli}
    H^2(M,\Z)\simeq v^\perp \hookrightarrow \tilde{H}(S,\Z)\simeq \tilde{\Lambda}.
\end{equation} 
Here, the first isometry is the Mukai morphism of Theorem \ref{thm:mukaimorphism}, and the isometry $\tilde{H}(S,\Z)\simeq \tilde{\Lambda}$ is arbitrary. Note that the $O(\tilde{\Lambda})$-orbit of the embedding \eqref{eq:naturalembeddingformoduli} is independent of the choice of isometry $\tilde{H}(S,\Z)\simeq \tilde{\Lambda}$.

Let $S$ be a K3 surface. Recall from Definition \ref{def:caldararuclass} that for a vector $v\in \N(S)$, we call the element $a_v\coloneqq \frac{v}{\dv(v)}\in A_{\N(S)}$ the C\u ald\u araru class of $v$, and that the element $\omega_v\in A_{T(S)}$ corresponding to $a_v$ via the natural isomorphism $A_{\N(S)}\simeq A_{T(S)}(-1)$ is called the transcendental C\u ald\u araru class of $v$.

\begin{proposition}\label{prop:birationaliffcaldararuclass}
    Let $S$ be a K3 surface for which there exists an embedding $U\subset \NS(S)$. Let $u,v\in \operatorname{N}(S)$ be primitive Mukai vectors with $u^2=v^2=2n-2,$ $n>1$. Then the moduli spaces $M(u)$ and $M(v)$ are birational if and only if there is a Hodge isometry $\psi:T(S)\simeq T(S)$ such that $\overline{\psi}(\omega_v) = \omega_u$.
\end{proposition}
\begin{proof}
    Suppose $M(u)$ and $M(v)$ are birational. Then by the Birational Torelli Theorem, there is a commutative diagram of Hodge isometries:
    \begin{equation}
        \xymatrix{
        v^\perp \ar[d] \ar^\simeq[r] & u^\perp \ar[d]\\
        \tilde{H}(S,\Z) \ar^\simeq_\Phi[r] & \tilde{H}(S,\Z).
        }
    \end{equation}
    Denote $\phi\coloneqq \Phi|_{\N(S)}: \N(S)\simeq \N(S)$ and $\psi\coloneqq \Phi|_{T(S)}: T(S)\simeq T(S)$, and note that $\psi$ is a Hodge isometry. Then $\phi(v) = \pm u$. Possibly replacing $\Phi$ by $-\Phi$, we may assume $\phi(v) = u$. We have $\overline{\phi}(a_v)=a_u$, hence $\overline{\psi}(\omega_v) = \omega_u$ by Lemma \ref{lem:nikulinorthogonal}.

    Conversely, suppose $\psi: T(S)\to T(S)$ is a Hodge isometry which satisfies $\overline{\psi}(\omega_v) = \omega_u$. Since $T(S)^\perp = \N(S)$, we may apply Lemma \ref{lem:extendisometryhyperbolic}, and extend $\psi$ to a Hodge isometry $\Psi:
    \tilde{H}(S,\Z)\to \tilde{H}(S,\Z)$. Denote $\phi\coloneqq\Psi|_{\N(S)}: \N(S)\simeq \N(S)$. 
    Since $\overline{\psi}(\omega_v) = \omega_u$, we have $a_{\phi(v)} = \overline{\phi}(a_{v}) = a_u$. By Lemma \ref{lem:eichlertransvections}, there is an isometry $\iota: \N(S)\simeq \N(S)$ which maps $\phi(v)$ to $u$ and which acts trivially on $A_{\N(S)}$. In particular, we can extend $\iota$, by the identity on $T(S)$, to a Hodge isometry  $\Gamma: \tilde{H}(S,\Z)\simeq \tilde{H}(S,\Z)$. Now $\Gamma\circ \Psi$ is a Hodge isometry of $\tilde{H}(S,\Z)$ which maps $v$ to $u$, and this implies by the Birational Torelli Theorem that $M(v)$ and $M(u)$ are birational.
\end{proof}

\begin{proposition}\label{prop:finemodulispacebirationaltohilbertscheme}
    Let $S$ be a K3 surface for which there exists an embedding $U\subset \NS(S)$. If $M$ is a fine(!) $2n$-dimensional moduli space of sheaves with Mukai vector $v\in \N(S)$, then $M$ is birational to $S^{[n]}$. 
\end{proposition}
\begin{proof}
    If $n=1$, then $M$ is a Fourier--Mukai partner of $S$ by the Derived Torelli Theorem \cite{Muk87,Orl03}. However, since $U\subset \NS(S)$, $S$ does not have any non-trivial Fourier--Mukai partners by \cite[Corollary 2.7(3)]{HLOY02}. Therefore, we have $M\simeq S$, as required.

    Now assume $n\geq 2$.
    By Theorem \ref{thm:explicitobstruction}, we have $\dv(v) = 1$. In particular, the C\u ald\u araru class of $v$ is $$a_v = \frac{v}{\dv(v)} = v = 0\in A_{\N(S)}.$$ If $u\in \N(S)$ is any other Mukai vector with $u^2 = 2n-2$ and  $\dv(u) = 1$, then $a_u = 0$ as well, hence $M$ is birational to $M(u)$ by Proposition \ref{prop:birationaliffcaldararuclass}. If we let $u = (1,0,1-n)$, then $M(u)\simeq S^{[n]}$, hence $M$ is birational to $S^{[n]}$.
\end{proof}

\begin{theorem}\label{thm:noderivedtorelli}
    Let $S$ be a K3 surface for which there exists an embedding $U\subset \NS(S)$. Let $M = M_H(v)$ be a moduli space of sheaves on $S$ of dimension $2n>2$. Consider the following three statements. 
    \begin{itemize}
        \item[i)] $M$ is a fine moduli space.
        \item[ii)] $M$ is birational to the Hilbert scheme $S^{[n]}$.
        \item[iii)] $M$ is derived equivalent to the Hilbert scheme $S^{[n]}$.
    \end{itemize}
    Then we have $$i) \iff ii) \implies iii).$$ 
   % Moreover, if there exists an algebraic class $\delta \in \N(S)\subset \tilde{H}(S,\Z)$ 
\end{theorem}

% The equivalence $ii)\iff iii)$ in Theorem \ref{thm:noderivedtorelli} was already noted by Beckmann in \cite{Bec22}. The addition of item $i)$ is a straightforward application of Theorem \ref{thm:explicitobstruction}, combined with Proposition \ref{prop:finemodulispacebirationaltohilbertscheme}.

\begin{proof}
    The implication $i)\implies ii)$ is Proposition \ref{prop:finemodulispacebirationaltohilbertscheme}. 
    We now prove that $ii)\implies i).$
    Suppose $M$ is birational to $S^{[n]}$. Let $u = (1,0,1-n)\in \N(S)$, so that $S^{[n]}\simeq M(u)$. Then the divisibility of $v$ in $\N(S)$ is equal to the divisibility of $u=(1,0,1-n)$ in $\N(S)$ by Proposition \ref{prop:birationaliffcaldararuclass}, hence $M$ is a fine moduli space by Corollary \ref{cor:fineiffdivone}. The equality of the divisibilities of $v$ and $u$ is also an easy consequence of the Birational Torelli Theorem, and does not depend on the existence of a primitive embedding $U\subset \NS(S)$. Indeed, by Theorem \ref{thm:birationaltorelli}, if $M$ and $S^{[n]}$ are birational, there is a Hodge isometry $\phi: \tilde{H}(S,\Z)\simeq \tilde{H}(S,\Z)$ which satisfies $\phi(u) = \pm v$. Since $\phi$ is an isometry, the divisibilities of $u$ and $v$ are equal. 

The implication $ii)\implies iii)$ is \cite{Hal21}.
    % The implication $iii)\implies ii)$ is \cite[Proposition 9.9]{Bec22}. The converse, $ii)\implies iii)$, is \cite{Hal21}.
\end{proof}

% As an immediate consequence of Theorem \ref{thm:noderivedtorelli}, we obtain the following corollary.

For $d\in \Z$, let $v_d=(0,H,d+1-g)$, where $H\subset S$ is a smooth, irreducible curve on a (not necessarily elliptic) K3 surface $S$, and assume $H$ is primitive in $\NS(S)$, and $H^2 = 2g-2\geq 2$. We write $$a_d\coloneqq a_{v_d}=\frac{v_d}{\dv(v_d)}$$ for the C\u ald\u araru class of $v_d$.

\begin{lemma}\label{lem:caldararuclassbm}
    We have $$a_d = \frac{H}{\dv(v_d)}\in A_{\N(S)}.$$ In particular, for $e\in \Z$, we have $a_d=a_e$ if and only if $\dv(v_d) = \dv(v_e)$.
\end{lemma}
\begin{proof}
    By definition, $$a_d = \frac{1}{\dv(v_d)}(0,H,d+1-g)=\frac{H}{\dv(v_d)} + (0,0,\frac{d+1-g}{\dv(v_d)}).$$ The result follows from the fact that $\dv(v_d)$ divides $d+1-g$. 
\end{proof}

\begin{theorem}\label{thm:birationalequivalencebeauville--mukaioverelliptick3surfaces}
    Let $S$ be a K3 surface for which there exists an embedding $U\subset \NS(S)$. Let $H\subset S$ be a smooth irreducible curve with $H^2 = 2g-2$ and with primitive class in $\NS(S)$. Let $d,e\in \Z$. Then the following are equivalent:
    \begin{itemize}
        \item[i)] $\olPic^d$ is birational to $\olPic^e$.
        \item[ii)] The obstruction classes of $\olPic^d$ and $\olPic^e$ have the same order. 
        \item[iii)] The Mukai vectors of $\olPic^d$ and $\olPic^e$ have the same divisibility. More precisely, we have $\gcd(\dv(H),d+1-g) = \gcd(\dv(H),e+1-g)$.
    \end{itemize}    
\end{theorem}
\begin{proof}
    Note that $\gcd(\dv(H),d+1-g)=\dv(v_d)$ for all $d\in \Z$.
    The equivalence $ii)\iff iii)$ follows from the fact that the order of $\alpha_d$ is equal to $\dv(v_d)$ for all $d\in \Z$ by Theorem \ref{thm:explicitobstruction}. 
    
    By Proposition \ref{prop:birationaliffcaldararuclass}, $\olPic^d$ and $\olPic^e$ are birational if and only if there exists a Hodge isometry $\psi:T(S)\simeq T(S)$ such that $\overline{\psi}(\omega_d) = \omega_e$. In this case, the orders of $\omega_d$ and $\omega_e$ in $A_{T(S)}$ are equal. Since the order of $\omega_d$ equals the order of $a_d$, which is $\dv(v_d)$, this means that $\dv(v_d) = \dv(v_e)$.
    
    Conversely, if $\dv(v_d)=\dv(v_e)$, then $a_d = a_e$ by Lemma \ref{lem:caldararuclassbm}, and $\olPic^d$ and $\olPic^e$ are birational by Proposition \ref{prop:birationaliffcaldararuclass}.
\end{proof}

\begin{remark}
    Theorem \ref{thm:birationalequivalencebeauville--mukaioverelliptick3surfaces} provides many examples of birational Beauville--Mukai systems $\olPic^d$ and $\olPic^e$ whose generic fibres are not isomorphic as $\olPic^0_\eta$-torsors. Indeed, let $D\subset \NS(S)$ be a primitive divisor with $\dv(D)>1$, and let $L\in \NS(S)$ be any ample divisor. Let $m$ be a large multiple of $\dv(D)$. Then the general member $H$ of the complete linear system $|mL + D|$ is a smooth and irreducible curve, and $\dv(H) = k\dv(D)$ for some $k\geq 1$. Moreover, we may assume $H$ is primitive, possibly after dividing by a divisor of $k$. For any $d,e\in \Z$ such that $\gcd(\dv(H),d+1-g)=\gcd(\dv(H),e+1-g)$, but such that $d\not\equiv e\pmod{\dv(H)}$, the Beauville--Mukai systems $\olPic^d$ and $\olPic^e$ are birational by Theorem \ref{thm:birationalequivalencebeauville--mukaioverelliptick3surfaces}, but their generic fibres are not isomorphic $\olPic^0_\eta$-torsors by Remark \ref{rem:whichdegreesgiveisomorphictorsors}.
\end{remark}

\subsection{Counterexamples to a Derived Torelli conjecture}

The straightforward generalisation of the Derived Torelli Theorem \ref{thm:introderivedtorelli} to hyperk\"ahler manifolds of \Kn-type was studied in \cite{KK24} and \cite{PR23}. The main goal of this section is to prove the following theorem, which provides counterexamples to this conjecture.

\begin{theorem}\label{thm:coexampleDerTorelli}
    For every $n\geq 3$, there exist moduli spaces of sheaves $M$ and $M'$ of dimension $2n$, on the same K3 surface, for which there is a Hodge isometry $T(M)\simeq T(M')$, but which are not derived equivalent.
\end{theorem}

In the proof of Theorem \ref{thm:coexampleDerTorelli}, we use the construction of the \textit{extended Mukai latttice} by Taelman and Beckmann \cite{Tae21,Bec22}. We briefly recall the definition here.
To any hyperk\"ahler manifold $M$ of K$3^{[n]}$-type, one can associate the extended Mukai lattice
\begin{equation}\label{eq:extendedmukailattice}
    \widetilde{\Lambda}_M\coloneqq H^2(M,\bbZ)\oplus U(-1),
\end{equation}
where the standard generators of $U(-1)$ are denoted $\alpha,\beta$. The Hodge structure on $\Lambda_{M'}$ is not the one induced by \eqref{eq:extendedmukailattice}. Instead, given any class $\delta_M\in H^2(M,\bbZ)$ satisfying $\delta_M^2=2-2n$ and $\dv(\delta_M)=2n-2$, one defines a \textit{twisted Hodge structure} of weight $2$ on $\widetilde{\Lambda}_M$ by $$\widetilde{\Lambda}_M^{2,0}\coloneqq \CC\cdot\left(\sigma+\frac{\sigma\cdot\delta_M}{2}\beta\right)$$ where $\sigma\in H^0(M,\Omega^2)$ is the symplectic form. Naturally, one defines $\widetilde{\Lambda}_M^{0,2}\coloneqq \overline{\widetilde{\Lambda}_M^{2,0}}$ and $\widetilde{\Lambda}_M^{1,1}=(\widetilde{\Lambda}_M^{2,0})^\perp$. In \cite[Theorem 9.2]{Bec22}, it is proved that any derived equivalence
$$\Db(M)\xrightarrow{\sim} \Db(M')$$
induces a Hodge isometry
$$\widetilde{\Lambda}_M \xrightarrow{\sim}\widetilde{\Lambda}_{M'}.$$

\begin{proof}[Proof of Theorem \ref{thm:coexampleDerTorelli}]
    Let $(S,H)$ be a polarized K3 surface with $\Pic(S)=\bbZ\cdot H$, $H^2=2g-2\geq 4$ and consider the moduli spaces $M=\olPic^{g-1}=M(0,H,0)$ and $M'=\olPic^g=M(0,H,1)$. Then $T(M)\simeq T(M')$. To prove that $M$ and $M'$ are not derived equivalent, we prove that their extended Mukai lattices are not Hodge isometric, more precisely they have non-isometric $(1,1)$-part. 

    Consider the Mukai lattice $\widetilde{H}(S,\bbZ)$, and fix $\delta\in T(S)$ the class such that $\frac{H-\delta}{2g-2}\in H^2(S,\bbZ)$, so that $\delta^2=2-2g$ and $\dv_{T(S)}(\delta)=2g-2$ (\ref{rmk:discriminantorthogonal}).
    Via the Mukai morphism, one gets $H^2(M,\bbZ)\simeq (0,H,0)^\perp=U\oplus T(S)\subset \widetilde{H}(S,\bbZ)$ and therefore $\NS(M)=U$. Similarly, one finds that $H^2(M',\bbZ)$ is generated by $T(S)$, $(0,0,1)$ and $(1,\frac{H-\delta}{2g-2},0)$, and therefore $\NS(M')=\langle (0,0,1),(2g-2,H,0)\rangle$.
    
    We now consider the extended Mukai lattices $\widetilde{\Lambda}_M$ (resp. $\widetilde{\Lambda}_{M'}$) with respect to the class $\delta_M=(0,\delta,0)$ (resp. $\delta_{M'}=(2g-2,H,1)$). A quick computation shows that $\widetilde{\Lambda}_{M}^{1,1}$ is generated by $\NS(M)$ and the two classes $2\alpha+\delta_M, \beta$. In particular, we get that
    $$\disc \widetilde{\Lambda}_{M}^{1,1}=4\cdot \disc \NS(M)=4.$$
    On the other hand, we have $\sigma\cdot \delta_{M'} = \sigma\cdot (2g-2,H,1) = 0,$ thus $\Lambda_{M'}^{1,1} = \Z\alpha\oplus \NS(M')\oplus \Z\beta \simeq U\oplus \NS(M').$ In particular, we get
    $$\disc \widetilde{\Lambda}_{M'}^{1,1}=\disc U\oplus\NS(M')= \disc(\NS(M'))=(2g-2)^2.$$
    Since $2g-2 \geq 4$, we have 
    \[
    \disc \widetilde{\Lambda}_{M}^{1,1} = 4 \neq (2g-2)^2 = \disc \widetilde{\Lambda}_{M'}^{1,1},
    \]
    we obtain that $\widetilde{\Lambda}^{1,1}_{M}$ and $\widetilde{\Lambda}^{1,1}_{M'}$ are not isometric. In particular, by \cite[Theorem 9.2]{Bec22}, $M$ and $M'$ are not derived equivalent.
\end{proof}

\printbibliography
\end{document}

%% file: preambule.tex
\DeclareMathOperator{\rk}{rk}

\DeclareMathOperator{\Hom}{Hom}
\DeclareMathOperator{\im}{Im}        %======== Care that \Im gives the imaginary part

\DeclareMathOperator{\Pic}{Pic}

\DeclareMathOperator{\Br}{Br}

\DeclareMathOperator{\SHom}{\mathcal{H\kern -1pt}\textit{om}} %For the local (sheaf) Hom -- kinda useless, but looks slightly better
\DeclareMathOperator{\SEnd}{\mathcal{E\kern -1pt}\textit{nd}} %For the local (sheaf) End -- kinda useless, but looks slightly better
\DeclareMathOperator{\SExt}{\mathcal{E\kern -1pt}\textit{xt}} %For the local (sheaf) Ext -- kinda useless, but looks slightly better
\DeclareMathOperator{\NS}{NS}

\DeclareMathOperator{\et}{\acute{e}t} % For "ét" as in étale

\DeclareMathOperator{\ch}{ch}
\DeclareMathOperator{\td}{td}
\DeclareMathOperator{\Coh}{\mathbf{Coh}}

\def\dar[#1]{\ar@<2pt>[#1]\ar@<-2pt>[#1]}

\newcommand\matQ{\mathbb{Q}}

\newcommand\matZ{\mathbb{Z}}

\newcommand\bbC{\mathbb{C}}

\newcommand\bbG{\mathbb{G}}

\newcommand\bbQ{\mathbb{Q}}

\newcommand\bbZ{\mathbb{Z}}

\newcommand\calA{\mathcal{A}}
\newcommand\calB{\mathcal{B}}
\newcommand\calC{\mathcal{C}}
\newcommand\calD{\mathcal{D}}
\newcommand\calE{\mathcal{E}}
\newcommand\calF{\mathcal{F}}
\newcommand\calG{\mathcal{G}}

\newcommand\calK{\mathcal{K}}
\newcommand\calL{\mathcal{L}}
\newcommand\calM{\mathcal{M}}

\newcommand\calO{\mathcal{O}}

\newcommand\calS{\mathcal{S}}

\newcommand\calU{\mathcal{U}}
\newcommand\calV{\mathcal{V}}

\newcommand{\fonction}[5]{\begin{array}{lrcl} 
#1: & #2 & \longrightarrow & #3 \\
    & #4 & \longmapsto & #5 \end{array}}

\newcommand{\functionstar}[4]{
\begin{array}{rcl} #1 &\longrightarrow &#2 \\ #3&\longmapsto &#4 \end{array}
}

\newcommand{\isomorphismstar}[4]{
\begin{array}{rcl} #1 &\overset{\sim}{\longrightarrow} &#2 \\ #3&\longmapsto &#4 \end{array}
}
    